\documentclass{amsart}
\usepackage[dvipdfmx]{graphicx}
\usepackage{amscd}
\usepackage{amsthm}
\usepackage{amssymb}
\usepackage{bm}
\usepackage{amsmath,amsfonts}
\usepackage{here}
\usepackage{url}
\usepackage{physics}
\usepackage{xcolor}
\usepackage{appendix}
\usepackage[
]{hyperref}
\numberwithin{equation}{section}

\theoremstyle{plain} 
\newtheorem{dfn}{Definition}[section]
\newtheorem{thm}[dfn]{Theorem}
\newtheorem{prop}[dfn]{Proposition}
\newtheorem{lem}[dfn]{Lemma}

\newtheorem*{Tak}{Takens' Last Problem}
\newtheorem{claim}{Claim}
\theoremstyle{remark}
\newtheorem{rem}[dfn]{Remark}

\title{Homoclinic tangency of codimension 2 and historic wandering domains}
\author{KODAI YAMAMOTO}
\date{}


\begin{document}

\begin{abstract}
    We construct a four-dimensional diffeomorphism exhibiting a homoclinic tangency of the largest codimension, which admits a historic wandering domain of positive Lebesgue measure.
    Every orbit in this wandering domain exhibits historic behavior, in the sense that time averages do not converge.
    This example shows that homoclinic tangencies of the largest codimension can still give rise to positive Lebesgue measure sets with non-convergent statistical behavior.
\end{abstract}

\maketitle

\section{Introduction}
In dynamical systems that exhibit chaotic behavior, it is often difficult to directly track the long-term evolution of individual trajectories.
It is therefore natural to investigate the system through its statistical properties.
For a dynamical system $f$ on a closed manifold $M$, the \emph{empirical measures} of the orbit of a point $x \in M$ are given by a sequence of measures
\begin{equation}\label{eq1}
    \frac{1}{n+1} \sum_{i=0}^{n} \delta_{f^{i}(x)}, \quad n=1,2,3,\dots,
\end{equation}
which describe the statistical distribution of the orbit up to time $n$.
Here, $\delta_{f^{i}(x)}$ denotes the Dirac measure supported at $f^{i}(x)$.
It is preferable if the sequence \eqref{eq1} converges to an ergodic measure. This is indeed true for Lebesgue almost every point for the cases of volume-preserving map and uniformly hyperbolic systems such as Axiom A systems.
However, it is also known that, for some dynamical systems and for many initial points $x$, the sequences \eqref{eq1} do not converge and exhibit statistically irregular behavior.
Ruelle \cite{Rue2001} referred to such behavior as \emph{historic behavior}.
For instance, Colli and Vargas \cite{CV2001} constructed an example of an affine horseshoe map with homoclinic tangencies in dimension two for which the set of points exhibiting historic behavior has positive Lebesgue measure.
Later, Kiriki, Nakano and Soma \cite{KNS2023} obtained a non-trivial extension of this result to dimension three using a blender-horseshoe with homoclinic tangencies; see also \cite[Remark 1.5]{KNS2023}.

Both constructions are carried out in the presence of a homoclinic tangency, that is, a non-transverse intersection between the stable and unstable manifolds of a hyperbolic set.
In this paper, we focus on the codimension of a homoclinic tangency, and in particular on the case of the largest codimension.
To investigate this problem, we construct an explicit four-dimensional diffeomorphism that has a homoclinic tangency of codimension two and whose set of points exhibiting historic behavior has positive Lebesgue measure.

In the next section, we explicitly construct a diffeomorphism $F \colon \mathbb{R}^{4} \to \mathbb{R}^{4}$ that admits a hyperbolic set $\Lambda$, whose stable and unstable manifolds exhibit a homoclinic tangency of codimension two.
For this diffeomorphism $F$, we establish the following as the main result of this paper.

\begin{thm}\label{main}
    There exist sequences of diffeomorphisms $F_{n}$ and $G_{n}$ converging to $F$ in the $C^{r}$-topology $(1 \leq r < \infty)$ such that the following hold:
    \begin{itemize}
        \item [\rm{(1)}] $F_{n}$ has a historic contracting wandering domain $D$.
        \item [\rm{(2)}] $G_{n}$ has a contracting wandering domain such that the sequence \eqref{eq1} converges to the periodic measure supported in a saddle periodic orbit.
    \end{itemize}
\end{thm}

Since all existing results treat only homoclinic tangencies of codimension one, the current progress on Takens' last problem does not yet extend to tangencies of higher codimension.
To make this distinction precise, we now introduce the notion of the codimension of a homoclinic tangency.
For a homoclinic tangency point $q$, \emph{the codimension of tangency} $c_{T}(q)$ is defined by
\begin{equation*}
  c_{T}(q) = \dim T_{q}M - \dim (T_{q}W^{s}(p) + T_{q}W^{u}(p)).
\end{equation*}
We say that the tangency is \emph{large} if $c_{T}(q) \geq 2$, and of \emph{the largest codimension} if $c_{T}(q) = \lfloor \dim M/2 \rfloor$ where $\lfloor x \rfloor$ is the maximal integer not greater than $x$.

Research on homoclinic tangencies of high codimension was first developed by Barrientos and Raibekas \cite{BR2017}. They showed that there exists a diffeomorphism exhibiting a $C^{2}$-robust homoclinic tangency of large (but not the largest) codimension.
More recently, Asaoka \cite{Asa2022} obtained corresponding $C^{2}$-robust results in the case of the largest codimension.
These results suggest that the arguments developed in this paper may extend to much more general settings.
In the present paper, we restrict ourselves to constructing a concrete example, leaving a broader investigation of these questions for future work.

\subsection{Preliminaries}
We begin by recalling some basic definitions on dynamical systems that will be used throughout this paper. For further details, the reader is referred to the textbooks by Wen \cite{LW2016} and Walters \cite{Wal}.

\subsubsection{Hyperbolic set and homoclinic points}
Let $M$ be a closed $C^\infty$-Riemannian manifold and $f \colon M \to M$ a $C^{r}$-diffeomorphism with $r \geq 1$. We define \emph{the stable set} of a point $x \in M$ by
\begin{equation*}
  W^{s}(x,f) = \qty{y \in M \mid \lim_{n \to \infty}d(f^{n}(x), f^{n}(y)) = 0},
\end{equation*}
where $d(\cdot , \cdot)$ denotes the Riemannian distance on $M$. Similarly, \emph{the unstable set} of $x$ is defined by $W^{u}(x,f) = W^{s}(x,f^{-1})$.
We say that a compact $f$-invariant set $\Lambda \subset M$ is \emph{hyperbolic} if for every $x \in \Lambda$ the tangent space $T_{x}M$ admits a continuous invariant splitting
\begin{align*}
    T_{x}M  = E^{s}(x)\oplus E^{u}(x)
\end{align*}
and there exist constants $C>0$ and $0 < \lambda < 1$ such that
\begin{align*}
  \|Df^{n}(x)v\| \leq C\lambda^{n}\|v\|\quad \text{for } v \in E^{s}(x), \ n \geq 0, \\
  \|Df^{-n}(x)v\| \leq C\lambda^{n}\|v\|\quad \text{for } v \in E^{u}(x), \ n \geq 0.
\end{align*}
The stable set $W^{s}(p,f)$ of a point $p$ in a hyperbolic invariant set $\Lambda$ is a $C^{r}$-injectively immersed submanifold of $M$.
We say $q \in M$ is a \emph{homoclinic point} of $p$ if
\begin{equation*}
    q \in W^{s}(p,f) \cap W^{u}(p,f) \setminus\{p\}.
\end{equation*}
A homoclinic point $q$ of $p$ is called \emph{transverse} if
\begin{equation*}
  T_{q}W^{s}(p,f) \oplus T_{q}W^{u}(p,f) = T_{q}M;
\end{equation*}
otherwise, it is called a \emph{homoclinic tangency point}.

\subsubsection{Statistical properties}
Let $\mathcal{P}(M)$ be the set of all Borel probability measures on $M$.
A sequence of probability measures $\{\mu_{n}\} \subset \mathcal{P}(M)$ is said to converge to $\mu$ in the \emph{weak $\ast$-topology} if for every continuous function $\varphi \colon M \to \mathbb{R}$,
\begin{equation*}
  \lim_{n \to \infty}\int \varphi d\mu_{n} = \int \varphi d\mu.
\end{equation*}
Since $M$ is a compact metric space, $\mathcal{P}(M)$ is also a compact metric space with respect to the weak $\ast$-topology.
We say that a point $x \in M$ or the forward orbit $\{f^{i}(x) \mid i \geq 0\}$ of $x$ has \emph{historic behavior} if \eqref{eq1} does not converge as $n \to \infty$ in the weak $\ast$-topology.

A \emph{non-trivial wandering domain} for $f$ is a connected open set $D \subset M$ satisfying the following conditions:
\begin{itemize}
    \item [(1)] $f^{i}(D) \cap f^{j}(D) = \emptyset$ for any integers $i,j \geq 0$ with $i \neq j$,
    \item [(2)] the union of $\omega$-limit sets  of all $x \in D$, $\cup_{x \in D} \omega(x,f)$, is not equal to a single periodic orbit.
\end{itemize}
A non-trivial wandering domain $D$ is said to be \emph{contracting} if the diameter of $f^{n}(D)$ converges to zero as $n \to \infty$.
We say that $f$ has a \emph{historic wandering domain} $D$ if every point $x$ of a wandering domain $D$ has historic behavior.

\subsection{Related works}
As part of an effort to develop a unified understanding of dynamical historicity, Takens posed the following open problem, now known as Takens' last problem.

\begin{Tak}[\cite{Tak2008}]
  Whether there are persistent classes of smooth dynamical systems such that the set of initial states which give rise to orbits with historic behavior has positive Lebesgue measure?
\end{Tak}

An affirmative answer to Takens' last problem has been obtained by means of homoclinic tangencies, which constitute a typical example of non-hyperbolic phenomena.
Let $\text{Diff}^{r}(M)$ denote the space of $C^{r}$-diffeomorphisms on a surface $M$. Newhouse \cite{New1979} showed that for any $f \in \text{Diff}^{2}(M)$ having a homoclinic tangency associated with a saddle periodic point, there exists an open set $\mathcal{N} \subset \text{Diff}^{2}(M)$ whose closure $\overline{\mathcal{N}}$ contains $f$ and such that every $g \in \mathcal{N}$ has a $C^{2}$-robust homoclinic tangency. Such an open set $\mathcal{N}$ is called a Newhouse domain.
On the other hand, in dimensions three and higher, Bonatti and Díaz \cite{BD1996} introduced blenders as a mechanism for creating robust non-hyperbolic phenomena in higher-dimensional dynamics. Using a special class of blenders, called blender-horseshoes, they showed that there exists an open set in $\text{Diff}^{1}(M)$ such that every diffeomorphism in this open set has a $C^{1}$ robust homoclinic tangency \cite{BD2012}.

Colli and Vargas \cite{CV2001} constructed a horseshoe map with a homoclinic tangency that has a historic contracting wandering domain.
Using the ideas of Colli and Vargas, Kiriki and Soma \cite{KS2017} showed that any Newhouse domain of $\text{Diff}^{r}(M)$ with $\dim M =2$ and $2 \leq r < \infty$, which contains a dense subset whose elements all have a historic contracting wandering domain.
This provides the first affirmative answer to Takens's last problem in the sense of a dense set.
More recently, this result has been extended to the $C^{\infty}$ and analytic categories by Berger and Biebler \cite{BB2023}.

Although the result of Kiriki and Soma \cite{KS2017} is based on Newhouse's result, the diffeomorphisms obtained there are not necessarily contained in an open neighborhood of the Colli--Vargas model $f \in \text{Diff}^{r}(M)$ with $\dim M =2$ and $2 \leq r < \infty$.
Recently, by using the pluripotent property---a notion introduced by Kiriki, Nakano and Soma \cite{KNS2024}, allowing statistical approximation of arbitrary prescribed orbit dynamics---Kiriki et al. \cite{KLNSV2025} proved that there exists a $C^{r}$-open neighborhood of the Colli–Vargas model $f \in \text{Diff}^{r}(M)$ with $\dim M =2$ and $2 \leq r < \infty$, which contains a dense subset whose elements all have a historic contracting wandering domain.
In higher dimensions $\dim M \geq 3$, Barrientos \cite{Bar2022} obtained an affirmative answer to Takens' last problem by using the result of \cite{KS2017} which reduces high-dimensional dynamics to an appropriate two-dimensional setting.
On the other hand, by employing a different method based on the pluripotent property, Kiriki, Nakano, and Soma \cite{KNS2024} showed that there exists a $C^{r}$-open neighborhood $(2 \leq r < \infty)$ of the Kiriki--Nakano--Soma model \cite{KNS2023}, contains a dense subset whose elements all have a historic contracting wandering domain.
Consequently, an affirmative answer to Takens' last problem has also been obtained for $\dim M \geq 3$.

\subsection{Plan of this paper}
We outline the structure of this paper.
In Section \ref{model}, we describe the construction of the diffeomorphism in Theorem \ref{main}.
This diffeomorphism has a homoclinic tangency of codimension two.
In Section \ref{cc}, we construct a return map near the homoclinic tangency in order to obtain the result of Theorem \ref{main}.
To ensure that two rectangles have a nonempty intersection, we follow the argument in \cite{Asa2022}.
In Section \ref{per}, we construct a sufficiently small perturbation of the map near the homoclinic tangency so as to connect the orbit returning to the tangency.
Through this perturbation, we obtain a contracting wandering domain that either exhibits historic behavior or has empirical measures converging to the periodic measure supported on a saddle periodic orbit, as stated in Theorem \ref{main}.

\section{Construction of the Map}\label{model}
In this section, we construct the diffeomorphism $F$ in Theorem \ref{main}. The map $F$ is defined so that it admits a blender-horseshoe and, moreover, that the blender-horseshoe has a homoclinic tangency of codimension two.

First, we define a two-dimensional dynamical system on a rectangle $B = [-2,2]^{2}$.
Let $1 < \lambda_{cu1} < 2 < \lambda_{cu0} < \lambda_{u}$ be constants. We define rectangles $V_{0}$ and $V_{1}$ in $B$ by
\begin{equation}\label{sq}
    \begin{split}
        V_{0} &= [-1 -2\lambda_{cu0}^{-1}, -1 + 2\lambda_{cu0}^{-1}] \times [-1-2\lambda_{u}^{-1}, -1+2\lambda_{u}^{-1}], \\
        V_{1} &= [1-3\lambda_{cu1}^{-1}, 1 + \lambda_{cu1}^{-1} ] \times [1-2\lambda_{u}^{-1}, 1+2\lambda_{u}^{-1}].
    \end{split}
\end{equation}
We assume that the constants $\lambda_{cu0}$ and $\lambda_{cu1}$ satisfy the following inequality:
\begin{equation}\label{con_cover1}
    \lambda_{cu0}^{-1} + \lambda_{cu1}^{-1} > 1.
\end{equation}
We define a map $f \colon V_{0} \sqcup V_{1} \to [-2,2]^{2}$ so that $f|_{V_i}$ for $i=0,1$ are unique affine maps preserving the positive directions of $x_{u}$-axis and $y_{u}$-axis and satisfying
\begin{equation*}
    f(V_{0}) = f(V_{1}) = [-2,2]^{2}.
\end{equation*}
The map $f$ is then explicitly given by
\begin{equation}\label{3_f1}
    f(x_{u}, y_{u}) =
    \begin{cases}
        (\lambda_{cu0}(x_{u} +1), \lambda_{u}(y_{u} +1) ) & \text{if $(x_{u}, y_{u}) \in V_{0}$}, \\
        (\lambda_{cu1}(x_{u} - 1)+1 , \lambda_{u}(y_{u} -1)) & \text{if $(x_{u}, y_{u}) \in V_{1}$}.
    \end{cases}
\end{equation}

Next, choose constants $0 < \lambda_{s} < \lambda_{cs0} < 1/2 < \lambda_{cs1} <1$, and define rectangles
\begin{align*}
    \widetilde{V}_{0} &= [-1-2\lambda_{s}, -1 + 2\lambda_{s}] \times [-1 - 2\lambda_{cs0}, -1+2\lambda_{cs0}], \\
    \widetilde{V}_{1} &= [1 - 2\lambda_{s}, 1 +2\lambda_{s}] \times [1 - 3\lambda_{cs1}, 1+\lambda_{cs1} ].
\end{align*}
We assume that the constants $\lambda_{cs0}$ and $\lambda_{cs1}$ satisfy the following inequality:
\begin{equation}\label{con_cover2}
    \lambda_{cs0} + \lambda_{cs1} > 1.
\end{equation}
We define the map $g \colon \widetilde{V}_{0} \sqcup \widetilde{V}_{1} \to [-2,2]^{2}$ in the same manner as $f$ so that
\begin{equation*}
    g(\widetilde{V}_{0}) = g(\widetilde{V}_{1}) = [-2,2]^{2}.
\end{equation*}
The map $g$ is then explicitly given by
\begin{equation}\label{3_f2}
    g(x_{s},y_{s}) =
    \begin{cases}
        (\lambda_{s}^{-1}(x_{s}+1), \lambda_{cs0}^{-1}(y_{s} + 1 )) & \text{if $(x_{s}, y_{s}) \in \widetilde{V}_{0}$}, \\
        (\lambda_{s}^{-1}(x_{s} + 1) , \lambda_{cs1}^{-1}(y_{s} - 1)+1 ) & \text{if $(x_{s}, y_{s}) \in \widetilde{V}_{1}$}.
    \end{cases}
\end{equation}

\begin{rem}
    In this paper, the intersection conditions of the projections of the renctangles $(V_{i})_{i=0,1}$ and $(\widetilde{V}_{i})_{i=0,1}$ onto the $x_{u}$ and $y_{s}$-axes play an essential role.
    By condition \eqref{con_cover1}, the rectangles $(V_{i})_{i=0,1}$ have overlapping projections on the $x_{u}$-axis.
    Moreover, by the condition \eqref{con_cover2}, the rectangles $(\widetilde{V}_{i})_{i=0,1}$ have overlapping projections on the $y_{s}$-axis.
\end{rem}

Define subsets $(\mathbb{V}_{i})_{i=0,1}$ and $(\widetilde{\mathbb{V}}_{i})_{i=0,1}$ of the four-dimensional box $\mathbb{B} = [-2,2]^{4}$ by
\begin{align*}
     \mathbb{V}_{0} = V_{0} \times [-2,2]^{2}, \quad \mathbb{V}_{1} = V_{1} \times [-2,2]^{2} ,\\
     \widetilde{\mathbb{V}}_{0} = [-2,2]^{2} \times \widetilde{V}_{0}, \quad \widetilde{\mathbb{V}}_{1} = [-2,2]^{2} \times \widetilde{V}_{1}.
\end{align*}
We then define a map $F \colon \mathbb{V}_{0} \sqcup \mathbb{V}_{1} \to [-2,2]^{4}$ as the product of $f$ \eqref{3_f1} and the inverse of $g$ \eqref{3_f2}:
\begin{equation}\label{2_f}
    F(x_{u}, y_{u}, x_{s}, y_{s}) =
    \begin{cases}
        (f(x_{u},y_{u}), (g|_{V_0})^{-1}(x_{s},y_{s}))&\text{ if $(x_{u}, y_{u}, x_{s}, y_{s})\in \mathbb{V}_{0}$},\\
        (f(x_{u},y_{u}), (g|_{V_1})^{-1}(x_{s},y_{s}))&\text{ if $(x_{u}, y_{u}, x_{s}, y_{s})\in \mathbb{V}_{1}$}.
    \end{cases}
\end{equation}
The map \eqref{2_f} can be expressed explicitly as
\begin{align}\label{2_f1}
    &F(x_{u}, y_{u}, x_{s}, y_{s})  \notag  \\ &=
    \begin{cases}
        (\lambda_{cu0}(x_{u} +1), \lambda_{u}(y_{u} +1) , \lambda_{s}x_{s}-1, \lambda_{cs0}y_{s}-1 ) \\ \hspace{20em} \text{if}\ (x_{u}, y_{u}, x_{s}, y_{s}) \in \mathbb{V}_{0}, \\
       (\lambda_{cu1}(x_{u} - 1) +1, \lambda_{u}(y_{u} -1) ,\lambda_{s}x_{s} + 1, \lambda_{cs1}(y_{s} - 1) +1) \\ \hspace{20em} \text{if} \ (x_{u}, y_{u}, x_{s}, y_{s}) \in \mathbb{V}_{1}.
    \end{cases}
\end{align}
The images of $\mathbb{V}_{0}$ and $\mathbb{V}_{1}$ under $F$ are given by
\begin{align*}
    F(\mathbb{V}_{0}) = \widetilde{\mathbb{V}}_{0}, \quad
    F(\mathbb{V}_{1}) = \widetilde{\mathbb{V}}_{1}.
\end{align*}
Moreover, the constants $\lambda_{cs0}$, $\lambda_{cs1}$, and $\lambda_{u}$ are assumed to satisfy
\begin{align}\label{2_con1}
     \lambda_{cs0}\lambda_{cs1}\lambda_{u}^{2} < 1
\end{align}
In addition, we choose the constants so that
\begin{equation}\label{claim_con}
    1 < \lambda_{cs0}\lambda_u^{2} < 2.
\end{equation}

\begin{rem}
    The condition \eqref{2_con1} ensures that the wandering domain that we construct is contracting. Since $\lambda_{s} < \lambda_{cs0} < \lambda_{cs1}$ and $\lambda_{cu1} < \lambda_{cu0} < \lambda_{u}$, it follows from \eqref{2_con1} that
    \begin{equation}\label{2_con2}
        \lambda_{cu0}\lambda_{cu1}\lambda_{s}^{2} < 1.
    \end{equation}
\end{rem}

\begin{figure}[H]
    \centering
    \scalebox{0.2}{\includegraphics{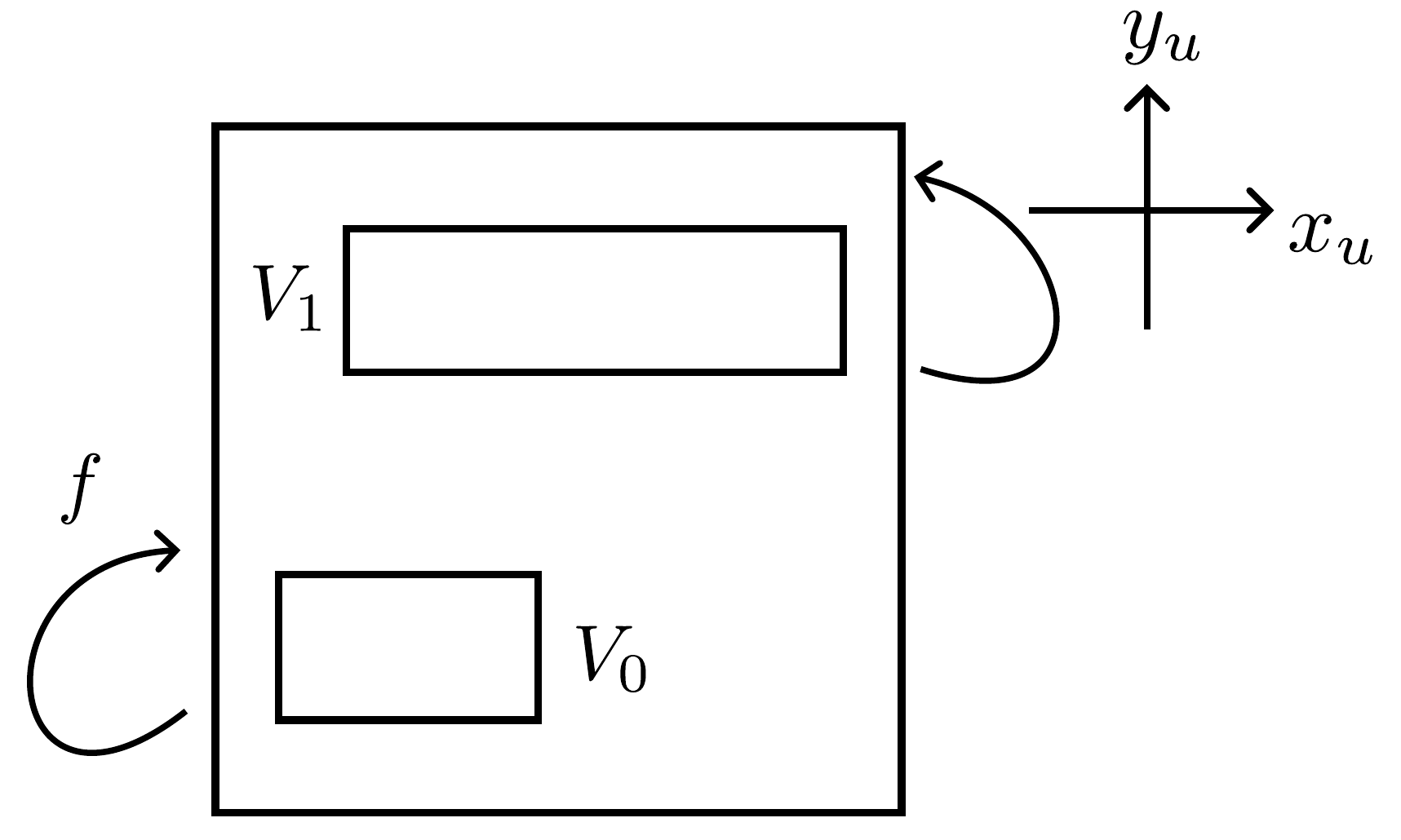}}
    \caption{}
    \label{pro_dyna}
\end{figure}

From the definition \eqref{2_f} of the diffeomorphism $F$, the closed invariant subset
\begin{equation*}
  \Lambda = \bigcap_{n \in \mathbb{Z}} F^{n}(\mathbb{V}_{0} \cup \mathbb{V}_{1})
\end{equation*}
is a hyperbolic set of $F$. It contains two fixed points of $F$,
\begin{align*}\label{fix}
    p =  \qty(-a_{cu}, -a_{u}, -a_{s}, -a_{cs}) , \quad
    q =  \qty(1, a_{u}, a_{s}, 1).
\end{align*}
where the constants $a_{cu}$, $a_{u}$, $a_{s}$ and $a_{cs}$ are given by
\begin{equation*}
    a_{cu} = \frac{1}{1-\lambda_{cu0}^{-1}},\quad a_{u} = \frac{1}{1-\lambda_{u}^{-1}}, \quad a_{s} = \frac{1}{1-\lambda_{s}},\quad a_{cs} = \frac{1}{1-\lambda_{cs0}}.
\end{equation*}

Let $0 < \delta < 1-2\lambda_{u}^{-1}$ be a constant, and consider the box
\begin{equation}\label{tan_dom}
    \mathbb{V}_{*} = \qty[-\delta,  \delta]^{2} \times [-2,2]^{2},
\end{equation}
which does not intersect $\mathbb{V}_{0} \sqcup \mathbb{V}_{1}$.
The map $F$ is then defined on $\mathbb{V}_{*}$ by
\begin{equation}\label{model_tan}
    F(x_{u}, y_{u}, x_{s}, y_{s})
    = \biggl(-x_{u}^{2} + \lambda_{*}^{-1}(x_{s} - a_{s}) +1  , -a_{1}y_{u}^{2} +\mu_{*} (y_{s} -1) +a_{u}  ,  a_{2}x_{u} , y_{u} \biggr),
\end{equation}
where $a_{1}$ and $a_{2}$ are parameters and  $a_s$, $a_u$, $\lambda_*$ and $\mu_*$ are constants.
The constants $\lambda_{*}$, $\mu_{*} > 0$ are taken to be sufficiently large.
We assume that the parameters $a_{1}$ and $a_{2}$ satisfy
\begin{equation}\label{2_par}
    a_{1} > \frac{1}{2(1-2\lambda_{u}^{-1})}, \quad 0 < a_{2} < \frac{2\lambda_{s}}{\delta}.
\end{equation}

The restriction of $F$ to $\mathbb{V}_{0} \sqcup \mathbb{V}_{1} \sqcup \mathbb{V}_*$ is injective under the parameter condition \eqref{2_par}. Since the images $F(\mathbb{V}_{0})$, $F(\mathbb{V}_{1})$ and $F(\mathbb{V}_{*})$ are mutually disjoint, the map $F$ can be extended to a diffeomorphism on $\mathbb{R}^{4}$, and we actually define it in this way.
The subsequent arguments do not depend on that extension.

By this construction, the hyperbolic set $\Lambda$ has a homoclinic tangency of codimension two.
Indeed, the stable and unstable manifolds of the fixed point $q$ contain, respectively,
\[
\qty{(1, a_{u})} \times [-2,2]^{2}, \quad
[-2,2]^{2} \times \qty{(a_{s}, 1)}
\]
and for the point $(0, 0, a_{s}, 1) \in [-2,2]^{2} \times \{( a_{s}, 1)\}$, we have
\begin{align*}
    F\qty(0,0,a_{s}, 1) = \qty(1, a_{u},0,0) \in \qty{(1, a_{u})} \times [-2,2]^{2}.
\end{align*}
Hence, the image of the latter under $F$ has a homoclinic tangency with the former, and the tangency is of codimension two.
Therefore, the hyperbolic set $\Lambda$ admits a homoclinic tangency of codimension two.

\section{Return to the Homoclinic Tangency}\label{cc}
In order to prove Theorem \ref{main}, we begin by analyzing how orbits of the map $F$ return to a neighborhood of the homoclinic tangency point.

\subsection{Unstable Boxes and Stable Boxes}
We denote by $|\underline{w}|=n$ the length of a code
\begin{equation*}
    \underline{w} = w(1) \cdots w(n) \in \{0,1\}^{n}
\end{equation*}
and define the associated $u$- and $s$-boxes by
\begin{align*}
  \mathbb{B}^{u}(\underline{w}) &= \qty{x \in \mathbb{B} \ | \ F^{i-1}(x) \in \mathbb{V}_{w(i)}, i=1, \ldots, |\underline{w}| }, \\
  \mathbb{B}^{s}(\underline{w}) &= \qty{x \in \mathbb{B} \ | \ F^{-i}(x) \in \mathbb{V}_{w(i)}, i=1, \ldots, |\underline{w}| }.
\end{align*}
For a code $\underline{w}$, we define
\begin{equation*}
    \mathbb{B}_*(\underline{w}) = \qty{x \in \mathbb{B}^{u}(\underline{w}) \ \bigg| \ F^{|\underline{w}|}(x) \in \mathbb{V}_{*} } \subset \mathbb{B}^{u}(\underline{w}).
\end{equation*}
Then
\begin{equation*}
    F^{|\underline{w}|}(\mathbb{B}_*(\underline{w})) = \mathbb{V}_{*} \cap \mathbb{B}^{s}( [\underline{w}]^{-1}),
\end{equation*}
where $[\underline{w}]^{-1} = w(n)\ldots w(1)$ denotes the code obtained by reversing the order of the symbols in $\underline{w}$.

In the above setting, let
\begin{equation*}
    c_{*}(\underline{w})=(x_{u}(\underline{w}),y_{u}(\underline{w}),0,0)
\end{equation*}
denote the center of the box $\mathbb{B}_*(\underline{w})$.
It then follows from definition of $F$ on $\mathbb{V}_{*}$ that $F^{|\underline{w}|}(c_{*}(\underline{w}))$ is the center of the box $F^{|\underline{w}|}(\mathbb{B}_*(\underline{w}))$.
In the next theorem, we construct codes for which the finite orbit
\[
c_{*}(\underline{w}), F(c_{*}(\underline{w})), \cdots, F^{|\underline{w}|}(c_{*}(\underline{w})), F^{|\underline{w}|+1}(c_{*}(\underline{w}))
\]
returns to a neighborhood of the homoclinic tangency point $q$.
The wandering domain in the main theorem will be obtained by connecting, through suitable perturbations, such orbits.

We now construct codes whose associated orbits return to a neighborhood of the homoclinic tangency with small error.
The precise statement is given in the following theorem.

\begin{thm}\label{main2}
    For any sufficiently large integer $L>0$, there exist codes $\underline{\alpha}^{(k)}$ and $\underline{\omega}^{(k+1)}$ $(k=1,2 \dots)$ satisfying the following conditions:
    \begin{itemize}
        \item [\rm{(1)}] There exists a constant $C>0$ such that
        \begin{equation*}
            |\underline{\alpha}^{(k)}| + |\underline{\omega}^{(k+1)}| < CLk.
        \end{equation*}
        \item [\rm{(2)}] For any codes $\underline{u}^{(k)}$ $(k = 1,2 \dots)$, define
        \begin{equation*}
            \underline{\gamma}^{(k)} := \underline{\alpha}^{(k)} \underline{u}^{(k)} [\underline{\omega}^{(k+1)}]^{-1}.
        \end{equation*}
        Let $c_{k}$ denote the center of the $u$-box $\mathbb{B}^u(\underline{\gamma}^{(k)})$.
        Then
        \begin{equation*}
            F^{|\underline{\gamma}^{(k)}|+1}(c_k) = c_{k+1} - (t_{k+1}, \widetilde{t}_{k+1},0,0)
        \end{equation*}
        where the real numbers $t_{k+1}$ and $\widetilde{t}_{k+1}$ satisfy
        \begin{equation*}
            |t_{k+1}|,|\widetilde{t}_{k+1}| < \lambda_{s}^{L(k+1)}.
        \end{equation*}
        \item [\rm{(3)}] The codes $\underline{\alpha}^{(k)}$ and $\underline{\alpha}^{(k+1)}$ differ starting from the $m_k$-th position, while the codes $\underline{\omega}^{(k)}$ and $\underline{\omega}^{(k+1)}$ differ starting from the $n_k$-th position. The integers $m_k$ and $n_k$ satisfy
        \begin{align*}
            m_{k} &< m_{k+1} < m_{k} + N_{u}, \\
            n_{k} &< n_{k+1} < n_{k} + N_{s},
        \end{align*}
        where $N_{u}$ and $N_{s}$ are positive integers independent of $L$ and $k$.
    \end{itemize}
\end{thm}

\begin{rem}
    For codes of length $n_k$, the distance between two $s$-boxes whose $n_k$-th symbols are different is given by $\lambda_{s}^{n_{k}}(1-2\lambda_{s})$.
    Since the integer $n_k$ satisfies condition (3), by choosing $L \gg N_s$ we obtain
    \begin{equation*}
        \lambda_{s}^{n_{k}}(1-2\lambda_{s}) \gtrsim \lambda_{s}^{N_{s}k} > \lambda_{s}^{Lk}.
    \end{equation*}
    This choice of $L$ produces an exponential separation of scales, so that the perturbation acts on a scale exponentially smaller than the separation between distinct $s$-boxes.
    Therefore, the perturbation defined in (2) can be localized to a neighborhood of the tangency, and outside this neighborhood the perturbed map can be chosen to coincide with the original map $F$. See Figure \ref{P}.
\end{rem}

\begin{figure}[H]
    \centering
    \scalebox{0.25}{\includegraphics{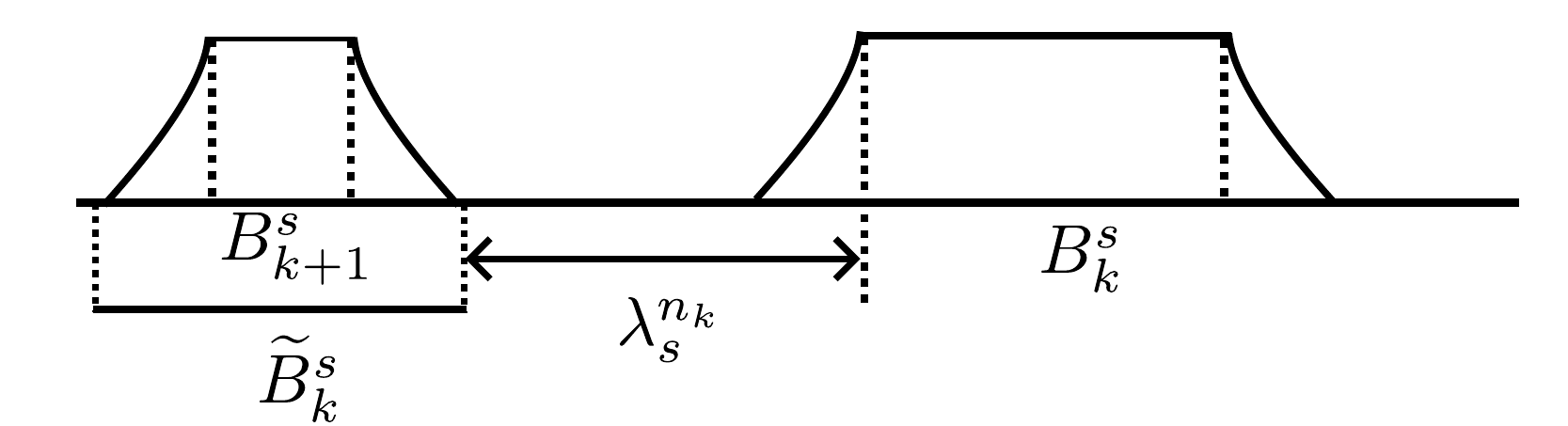}}
    \caption{}
    \label{P}
\end{figure}

We provide the proof of theorem \ref{main2} in the next Section.

\subsection{Linked rectangles}\label{C}
The map $F$ defined in \eqref{2_f} of Section \ref{model} has a hyperbolic invariant set $\Lambda$. The set $\Lambda$ is given as the product of Contor sets $\Lambda^{u}$ and $\Lambda^{s}$ in the $x_{u}y_{u}$-plane and the $x_{s}y_{s}$-plane, respectively.
Each of the Cantor sets $\Lambda^{u}$ and $\Lambda^{s}$ is given as the forward invariant set of the expanding maps $f$ and $g$ defined in \eqref{3_f1} and \eqref{3_f2}, respectively:
\begin{align}
    \Lambda^{u}(V_{0} \sqcup V_{1}, f) &= \bigcap_{n \geq 0}f^{-n}(V_{0} \sqcup V_{1}), \label{cantor1} \\
    \Lambda^{s}(\widetilde{V}_{0} \sqcup \widetilde{V}_{1}, g) &= \bigcap_{n \geq 0}g^{-n}(\widetilde{V}_{0} \sqcup \widetilde{V}_{1}).\label{cantor2}
\end{align}
In the following, in preparation for the proof of Section \ref{CV_key}, we establish several lemmas needed in the argument.

First, we fix positive constants $\lambda_{*}, \mu_{*} >0$ in \eqref{model_tan} so that
\begin{equation}\label{star}
    \lambda_{*}^{-1} < \frac{-2 + \lambda_{cu0}^{-1} + \lambda_{cu1}^{-1}a_{u}}{4} , \quad \mu_{*}^{-1} < \frac{-2 + \lambda_{cs1} + \lambda_{cs1}a_{s}}{4}.
\end{equation}
We then define a map $f_{*} \colon \mathbb{R}^{2} \to \mathbb{R}^{2}$ from the $x_{s}y_{s}$-plane to the $x_{u}y_{u}$-plane by
\begin{equation}\label{3_ftan}
    f_{*}(x_{s}, y_{s}) = (\lambda_{*}^{-1}(x_{s} -a_{s}) +1   , \mu_{*}(y_{s}-1)+a_{u}  ).
\end{equation}
This map is defined as the projection of $F$ in \eqref{model_tan} onto the $x_{s}y_{s}$-plane after restricting it to $x_{u} = y_{u} =0$.

Next, we define the rectangle $R_{*}$ in the $x_{s}y_{s}$-plane by
\begin{equation}\label{rec}
    R_{*} =  \qty[-2,2] \times \qty[1 -\mu_{*}^{-1}(2-a_{u}) ,  1 + \mu_{*}^{-1}(2-a_{u})].
\end{equation}
Let $P,Q \colon \mathbb{R}^{2} \to \mathbb{R}$ denote the projections onto the first and second coordinates, respectively.
We first note that the following facts follow immediately from the setting, but we verify them here for later use.

\begin{lem}\label{lemA}
The map $f_{*}$ defined in \eqref{3_ftan} and the rectangle $R_{*}$ defined in \eqref{rec}
satisfy the following relations:
\begin{align*}
    Q(R_{*}) \subset \mathrm{Int}\,(Q(\widetilde{V}_{1})), \quad P(R_{*}) = [-2,2], \\
    P(f_{*}(R_{*})) \subset \mathrm{Int}\,(P(V_{1})), \quad Q(f_{*}(R_{*})) = [-2,2].
\end{align*}
Moreover, for any $n \ge 0$, the rectangle $R_{*,n}$ and the map $f_{*,n}$ defined by
\begin{equation*}
    R_{*,n+1}
    = g\bigl(f_{*,n}^{-1}(V_{1}) \cap R_{*,n} \cap \widetilde{V}_{1}\bigr),
    \quad
    f_{*,n+1} = f \circ f_{*,n} \circ g^{-1}
\end{equation*}
satisfy the same relations, that is,
\begin{equation}\label{link_con}
    \begin{split}
        Q(R_{*,n}) \subset \mathrm{Int}\,(Q(\widetilde{V}_{1})), \quad P(R_{*,n}) = [-2,2], \\
        P(f_{*,n}(R_{*,n})) \subset \mathrm{Int}\,(P(V_{1})), \quad
        Q(f_{*,n}(R_{*,n})) = [-2,2].
    \end{split}
\end{equation}
\end{lem}
\begin{proof}
    From condition \eqref{star}, a direct computation implies the inequality
    \begin{equation*}
        \frac{-2 + \lambda_{cu0}^{-1} + \lambda_{cu1}^{-1}a_{u}}{4}<\frac{\lambda_{cu1}^{-1}(1-\lambda_{s})}{1-2\lambda_{s}}.
    \end{equation*}
    As a consequence, comparing the endpoints of the intervals $Q(R_{*})$ and $Q(\widetilde{V}_{1})$, we obtain
    \begin{equation*}
        1 -3\lambda_{cs0} < 1 -\mu_{*}^{-1}(2-a_{u}), \quad 1 + \mu_{*}^{-1}(2-a_{u}) < 1 +  \lambda_{cs1}.
    \end{equation*}
    Similarly, comparing $P(f_{*}(R_{*}))$ with $P(V_{1})$, condition \eqref{star} implies
    \begin{equation*}
        \frac{-2 + \lambda_{cs1} + \lambda_{cs1}a_{s}}{4} < \frac{\lambda_{cs1}(1-\lambda_{u}^{-1})}{1-2\lambda_{u}^{-1}}.
    \end{equation*}
    This yields
    \begin{equation*}
        1 -3\lambda_{cu0}^{-1} < 1 -\lambda_{*}^{-1}(2-a_{s}), \quad 1 + \lambda_{*}^{-1}(2-a_{s}) < 1 +  \lambda_{cu1}^{-1}.
    \end{equation*}
    These inequalities imply the claimed inclusion relations, and therefore $R_{*}$ and $f_{*}$ satisfy the desired properties.

    By the definition of the rectangle $R_{*}$, we have
    \begin{equation}\label{hougan}
        P(f_{*}(R_{*})) \subset \text{Int}\, P(V_{1}), \ Q(R_{*}) \subset \text{Int}\, Q(\widetilde{V}_{1}).
    \end{equation}
    In particular, \eqref{hougan} ensures that the intersection defining the next rectangle is nonempty.
    \begin{align*}
        &f_{*}^{-1}(V_{1}) \cap R_{*} \cap \widetilde{V}_{1} \\
        &= \qty[1-2\lambda_{s}, 1 + 2\lambda_{s}] \times \qty[1 + \mu_{*}^{-1}(1-2\lambda_{u}^{-1}-a_{u}), 1+\mu_{*}^{-1}(1+2\lambda_{u}^{-1}-a_{u})].
    \end{align*}
    We define the rectangle $R_{*,1}$ by
    \begin{equation*}
        R_{*,1} = g(f_{*}^{-1}(V_{1}) \cap R_{*} \cap \widetilde{V}_{1}).
    \end{equation*}
    The rectangle $R_{*,1}$ satisfies
    \begin{align*}
        P(R_{*,1}) = P(g(\widetilde{V}_{1})) = [-2,2], \\
        Q(f \circ f_{*} \circ g^{-1}(R_{*,1})) = Q(f(V_{1})) = [-2,2].
    \end{align*}
    Moreover, since the constant $\lambda_{u}\lambda_{cs1}>1$, we obtain
    \begin{equation*}
        Q(R_{*,1}) \subset \text{Int}\, Q(\widetilde{V}_{1}), \ P(f \circ f_{*} \circ g^{-1}(R_{*,1})) \subset \text{Int}\,P(V_{1}).
    \end{equation*}
    Therefore, \eqref{link_con} holds for the pair $(g(f_{*}^{-1}(V_{1}) \cap R_{*} \cap \widetilde{V}_{1}), f \circ f_{*} \circ g^{-1})$ which corresponds to the case $n=0$. See figure \ref{R}.
    For $n \geq 1$, assume that \eqref{link_con} holds for the pair $(R_{*,n}, f_{*,n})$.
    Recall that the pair $(R_{*,n}, f_{*,n})$ is defined by
    \begin{equation*}
        R_{*,n} = g(f_{*,n-1}^{-1}(V_{1}) \cap R_{*,n-1} \cap \widetilde{V}_{1}), \quad f_{*,n} = f \circ f_{*,n-1} \circ g^{-1}.
    \end{equation*}
    In particular, the intersection $f_{*,n}^{-1}(V_{1}) \cap R_{*,n} \cap \widetilde{V}_{1}$ is nonempty.
    We then define
    \begin{equation*}
        R_{*,n+1} = g(f_{*,n}^{-1}(V_{1}) \cap R_{*,n} \cap \widetilde{V}_{1}), \quad f_{*,n+1} = f \circ f_{*,n} \circ g^{-1}.
    \end{equation*}
    The same argument as in the case $n=0$ shows that \eqref{link_con} also holds for the pair $(R_{*,n+1}, f_{*,n+1})$.
\end{proof}

Thus \eqref{link_con} holds throughout the iterative construction and will be
used in Section \ref{CV_key}.

\begin{figure}[H]
    \centering
    \scalebox{0.19}{\includegraphics{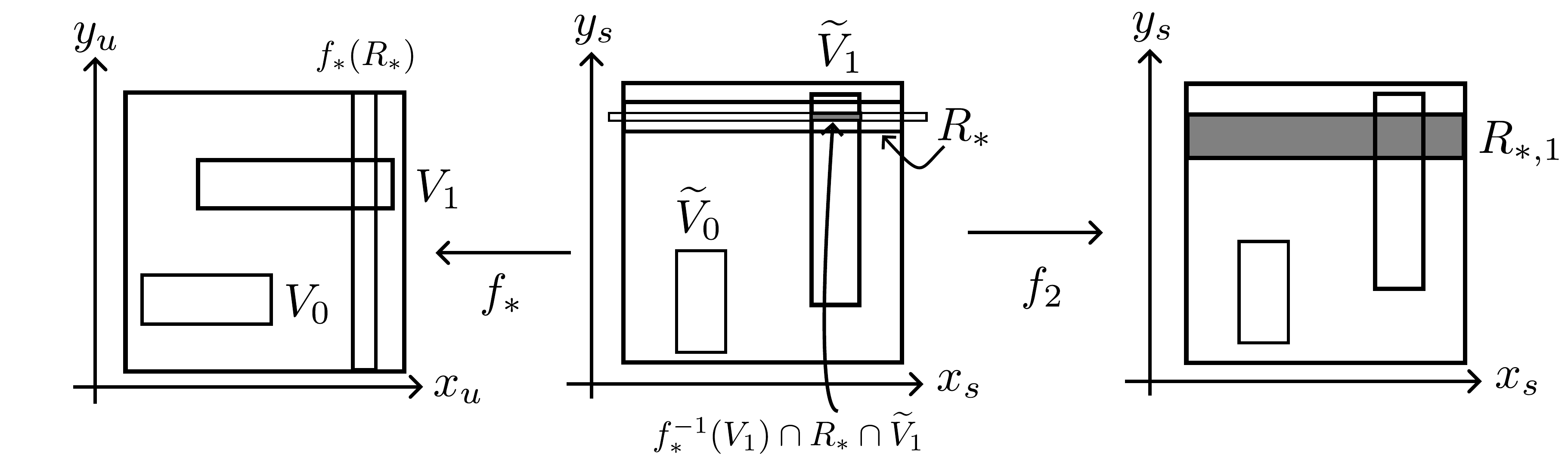}}
    \caption{}
    \label{R}
\end{figure}

\subsection{Proof of Theorem \ref{main2}}\label{CV_key}
In what follows, for $m \in \mathbb{N}$, we use the box norm $|\cdot|$ on the $m$-dimensional Euclidean space $\mathbb{R}^{m}$, that is,
\begin{equation*}
    |(x_{1},x_{2}, \ldots,  x_{m})| = \max \{|x_{1}|, |x_{2}|, \ldots , |x_{m}|\}.
\end{equation*}
We now introduce the two-dimensional $u$- and $s$-boxes in the $x_{u}y_{u}$-plane and prepare for the proof of theorem \ref{main2}.

Let $\pi_{u} \colon \mathbb{R}^{4} \to \mathbb{R}^{2}$ be the projection defined by
\begin{equation*}
    \pi_{u}(x_{u},y_{u},x_{s},y_{s}) = (x_{u},y_{u}).
\end{equation*}
For a code $\underline{w} \in \{0,1\}^{n}$, we define the $u$-box $B^{u}(\underline{w})$ by
\begin{equation}\label{box}
    B^{u}(\underline{w}) = \pi_{u} (\mathbb{B}^{u}(\underline{w})).
\end{equation}
We define the Cantor set $\Lambda^{u} \subset \mathbb{R}^{2}$ by
\begin{equation}\label{box2}
    \Lambda^{u} = \bigcap_{n \geq 1} \bigcup_{\underline{w} \in \{0,1\}^{n}} B^{u}(\underline{w}).
\end{equation}
Moreover, by the definition of the map $F$ in \eqref{2_f}, we have
\begin{equation*}
    F(0, 0, x_{s}, y_{s}) = (x_{s}, y_{s}, 0, 0)
\end{equation*}
for every $(0, 0, x_{s},  y_{s}) \in \mathbb{V}_{*}$.
Thus we define $s$-box $B^{s}(\underline{w})$ by
\begin{equation}\label{box3}
    B^{s}(\underline{w}) = \pi_{u} \circ F\qty(\qty(\{(0,0)\} \times [-2,2]^{2}) \cap \mathbb{B}^{s}(\underline{w})).
\end{equation}
We then define the Cantor set $\Lambda^{s} \subset \mathbb{R}^{2}$ by
\begin{equation}\label{box4}
    \Lambda^{s} = \bigcap_{n \geq 1} \bigcup_{\underline{w} \in \{0,1\}^{n}} B^{s}(\underline{w}).
\end{equation}

\begin{rem}
    The set $\Lambda^{u}$ defined in \eqref{box2} coincides with the forward invariant set $\Lambda^{u}(V_{0} \sqcup V_{1}, f)$ of $f$ in \eqref{cantor1}.
    Similarly, the set $\Lambda^{s}$ defined in \eqref{box4} coincides with the image $f_{*}(\Lambda^{s}(\widetilde{V}_{0} \sqcup \widetilde{V}_{1}, g))$ of the forward invariant set of $g$ in \eqref{cantor2} under the map $f_{*}$.
\end{rem}

In what follows, we describe the construction of the codes required for the proof of theorem \ref{main2}, using the $u$-boxes and $s$-boxes introduced in \eqref{box} and \eqref{box3}.
To this end, we first consider pairs of $u$- and $s$-boxes that intersect.
We then show that, by applying a small perturbation to the $s$-box in such a pair, one can obtain a new intersecting pair.
By iterating this procedure, we construct the codes necessary for theorem \ref{main2}.

Let $B^{u}(\underline{a})$ and $B^{s}(\underline{w})$ denote the $u$-box and $s$-box determined by the codes $\underline{a} \in \{0,1\}^{m}$ and $\underline{w} \in \{0,1\}^{n}$, respectively.
For a $u$-box $B^{u}(\underline{a})$ and an $s$-box $B^{s}(\underline{w})$ in $\mathbb{R}^{2}$, we define the norm of each box, for $i \in \{u,s\}$, by
\begin{equation*}
    |B^{i}(\underline{v})| = \max \{|P(B^{i}(\underline{v}))|,\  |Q(B^{i}(\underline{v}))|\}.
\end{equation*}
where $\underline{v}$ denotes the corresponding code.
We say that a pair $(B^{u}(\underline{a}), B^{s}(\underline{w}))$ of a $u$-box and an $s$-box is a \emph{linked pair} if $B^{u}(\underline{a}) \cap B^{s}(\underline{w}) \neq \emptyset$ and
\begin{equation*}
    P(B^{s}(\underline{w})) \subset \mathrm{Int}\,P(B^{u}(\underline{a})) \quad \text{and} \quad Q(B^{u}(\underline{a})) \subset \mathrm{Int}\,Q(B^{s}(\underline{w})).
\end{equation*}

\begin{rem}\label{lin_pair}
    For each $i \in \{0,1\}$, the rectangles $(V_{i})_{i=0,1}$ and $(\widetilde{V}_{i})_{i=0,1}$ satisfy
    \begin{equation*}
        Q(V_{i}) \subset (-2,2), \quad P(\widetilde{V}_{i}) \subset (-2,2).
    \end{equation*}
    Moreover, by lemma \ref{lemA}, we see that $B^{u}(1)$ and $B^{s}(1)$ form a linked pair.
    In particular, each $R_{*,n}$ produced by the iterative construction above gives rise to a linked pair of $u$- and $s$-boxes.
\end{rem}

Next, we say that a linked pair $(B^{u}(\underline{a}), B^{s}(\underline{w}))$ is \emph{$\kappa$-proportional} if there exists a constant $\kappa \in (0,1)$, independent of $B^{u}$ and $B^{s}$, such that
\begin{equation*}
  \kappa|B^{u}(\underline{a})| <  |B^{s}(\underline{w})| \leq |B^{u}(\underline{a})|.
\end{equation*}
We say that two $u$-boxes $B^{u}_{1}$ and $B^{u}_{2}$ are \emph{related} if one corresponds to a code obtained from the other by adding one additional symbol.
In the same manner, we define two $s$-boxes $B^{s}_{1}$ and $B^{s}_{2}$ to be related.

The following lemma shows that, for the linked pair $(B^{u}(1), B^{s}(1))$ of a $u$-box and an $s$-box, the related pairs of $u$- and $s$-boxes contained in each of them give rise to two new linked pairs after applying a sufficiently small perturbation.
This provides a fundamental tool in the construction of the codes.

Here we define in advance suitable subsets of the corresponding $u$-boxes and $s$-boxes in order to exploit the blender property in the following lemma.
First, choose a constant $c >0$ so that
\begin{equation*}
    \lambda_{cs0} + \lambda_{cs1} > \frac{1}{1-c}.
\end{equation*}
We then define the rectangle $A$ by
\begin{equation}\label{super}
    A = [-a_{cu} + c, 1-c] \times [-2,2].
\end{equation}
Next, define rectangles $(R_{i})_{i=0,1}$ by
\begin{align*}
    R_{0} &= [-1 + \lambda_{cu0}^{-1}(-a_{cu} +c), -1 + \lambda_{cu0}^{-1}(1-c)] \times [-1 -2\lambda_{u}^{-1}, -1 + 2\lambda_{u}^{-1}], \\
    R_{1} &= [1 + \lambda_{cu1}^{-1}(-a_{cu} +c), 1 + \lambda_{cu1}^{-1}(1-c)] \times [1-2\lambda_{u}^{-1}, 1+\lambda_{u}^{-1}].
\end{align*}
For a $u$-box $B^{u}(\underline{a})$, we define its subset
\begin{equation*}
    R^{u}(\underline{a}) = \qty{ x \in A \mid f^{i-1}(x) \in R_{a(i)},\ i=1,\dots |\underline{a}|}.
\end{equation*}
The construction for $s$-boxes is analogous. We set
\begin{equation*}
    \widetilde{A} = [-2,2] \times [-a_{cs} + c , 1-c]
\end{equation*}
and define rectangles $(\widetilde{R}_{i})_{i=0,1}$ by
\begin{align*}
    \widetilde{R}_{0} &= [-1 -2\lambda_{s}, -1 + 2\lambda_{s}] \times [-1 + \lambda_{cs0}(-a_{cs} +c), -1 + \lambda_{cs0}(1-c)], \\
    \widetilde{R}_{1} &= [1-2\lambda_{s}, 1+\lambda_{s}] \times [1 + \lambda_{cs1}(-a_{cs} +c), 1 + \lambda_{cs1}(1-c)].
\end{align*}
For an $s$-box $B^{s}(\underline{w})$, we define its subset
\begin{equation*}
    R^{s}(\underline{a}) = \qty{ x \in A \mid g^{-i} \circ f^{-1}_{*}(x) \in \widetilde{R}_{a(i)},\ i=1,\dots |\underline{w}|}.
\end{equation*}
Finally, we choose the rectangles $Z$ and $\widetilde{Z}$ so that, for any $i \in \{0,1\}$, the following conditions.
\begin{align*}
    Q(R_{i}) \subset \text{Int}Q(Z), \ P(f(R_{i})) \subset \text{Int}P(Z), \ Q(f(R_{i})) = Q(Z), \\
    P(\widetilde{R}_{i}) \subset \text{Int}P(\widetilde{Z}), \ Q(f(\widetilde{R}_{i})) \subset \text{Int}Q(\widetilde{Z}), \ P(f(\widetilde{R}_{i})) = P(\widetilde{Z}).
\end{align*}
Morover, the family $(P(R_{i}))_{i=0,1}$ covers $P(Z)$ and the family $(Q(\widetilde{R}_{i}))_{i=0,1}$ covers $Q(\widetilde{Z})$.

\begin{lem}\label{ll}
    There exist positive constants $N_{u}$ and $N_{s}$ satisfying the following:
    For any $\epsilon >0$, there exist a vector $\Delta$ in $\mathbb{R}^{2}$ with $|\Delta| < \epsilon$ and sequences of codes $(\underline{a}^{(k)})_{k \geq 1}$ and $(\underline{w}^{(k)})_{k \geq 1}$ such that, for every $k \geq 1$, the following hold:
    \begin{itemize}
        \item [\rm{(1)}] The $u$-boxes $B^{u}_{k}$ $(k=1,2,\dots)$ are pairwise disjoint and the $s$-boxes $B^{s}_{k}$ $(k=1,2,\dots)$ are pairwise disjoint.
        \item [\rm{(2)}] Let $B^{u}_{k}$ and $B^{s}_{k}$ denote the $u$-box and $s$-box determined by the codes $\underline{a}^{(k)}$ and $\underline{w}^{(k)}$, respectively.
        Then $(\Delta + B^{s}_{k}, B^{u}_{k})$ is a linked $\lambda_{cu0}^{-1}$-proportional pair.
        \item [\rm{(3)}] The codes $\underline{a}^{(k)}$ and $\underline{w}^{(k)}$ satisfy
        \begin{align*}
            |\underline{a}^{(k)}| &< |\underline{a}^{(k+1)}| \leq |\underline{a}^{(k)}| + N_{u}, \\
            |\underline{w}^{(k)}| &< |\underline{w}^{(k+1)}| \leq |\underline{w}^{(k)}| + N_{s}.
        \end{align*}
    \end{itemize}
\end{lem}
\begin{proof}
    We begin with part (1), where we inductively construct the $s$-boxes and $u$-boxes satisfying the desired properties.
    Consider first the case $k=1$.
    By remark \ref{lin_pair}, the boxes $B^{s}(1)$ and $B^{u}(1)$ form a linked pair.
    This initial linked pair, together with the preservation of the linked condition, allows us to construct a sequence of linked pairs inductively.
    Hence, the $u$-box and $s$-box whose symbols consist entirely of $1$'s,
    \begin{equation*}
        \bar{B}^{u}_{1}= B^{u}(1^{n_{1}-1}), \ \bar{B}^{s}_{1}= B^{s}(1^{m_{1}-1})
    \end{equation*}
    where $1^{k}$ denotes the symbol sequence consisting of $k$ consecutive symbols $1$, form a linked pair.
    Moreover, we choose the codes so that they satisfy
    \begin{equation}\label{ll_1}
        \frac{\lambda_{cs0}^{2}}{2}\epsilon \leq |\bar{B}^{s}_{1}| < \frac{\lambda_{cs0}}{2}\epsilon, \quad  \lambda_{cu0}^{-1}|\bar{B}^{u}_{1}| < |\bar{B}^{s}_{1}| \leq |\bar{B}^{u}_{1}|.
    \end{equation}
    Next, we examine the condition under which $(\delta_{1} + \bar{B}^{s}_{1}) \cap \bar{B}^{u}_{1}\neq \emptyset$.
    We define
    \begin{equation*}
        \bar{I}_{1} = \qty{ \delta \in \mathbb{R}^{2} \mid (\delta + \bar{B}^{s}_{1}) \cap \bar{B}^{u}_{1} \neq \emptyset }.
    \end{equation*}
    In particular, since $\bar{B}^{s}_{1}$ and $\bar{B}^{u}_{1}$ form a linked pair, we have $(0,0) \in \bar{I}_{1}$.
    Moreover, each side length of $\bar{I}_{1}$ is given by
    \begin{equation*}
        |P(\bar{I}_{1})| = |\bar{B}^{u}_{1}| + |P(\bar{B}^{s}_{1})|, \quad |Q(\bar{I}_{1})| = |\bar{B}^{s}_{1}| + |Q(\bar{B}^{u}_{1})|.
    \end{equation*}
    Furthermore, by \eqref{ll_1}, we have
    \begin{align}\label{ll_3}
        |\bar{B}^{u}_{1}| < \lambda_{cu0}|\bar{B}^{s}_{1}| < \frac{\lambda_{cs0}\lambda_{cu0}}{2}\epsilon.
    \end{align}
    Combining \eqref{ll_1} and \eqref{ll_3}, and using the condition in \eqref{2_con1}, the side lengths of $\bar{I}_{1}$ satisfy
    \begin{align*}
        |P(\bar{I}_{1})|, \ |Q(\bar{I}_{1})| <
        |\bar{B}^{s}_{1}| + |\bar{B}^{u}_{1}| &< \frac{\lambda_{cs0}}{2}\epsilon + \frac{\lambda_{cs0} \lambda_{cu0}}{2}\epsilon \\
        &< \epsilon.
    \end{align*}
    Hence we obtain $\bar{I}_{1} \subset (-\epsilon, \epsilon)^{2}$.
    In particular, each coordinate of the perturbation $\delta_{1}$ is bounded above by $\epsilon$.
    Next, in order to continue the inductive construction, we take the related pairs of $\bar{B}^{s}_{1}$ and $\bar{B}^{u}_{1}$, denoted $\widetilde{B}^{s}_{1}, B^{s}_{1}$ and $\widetilde{B}^{u}_{1}, B^{u}_{1}$, respectively. Since a related pair consists of the two boxes obtained by extending the code by one symbol, we may label these two boxes so that
    \begin{equation*}
        \max P(\widetilde{B}^{s}_{1}) < \min P(B^{s}_{1}), \quad \max Q(\widetilde{B}^{u}_{1}) < \min Q(B^{u}_{1}).
    \end{equation*}
    See Figure \ref{us-box}.

    \begin{figure}[H]
      \centering
      \scalebox{0.17}{\includegraphics{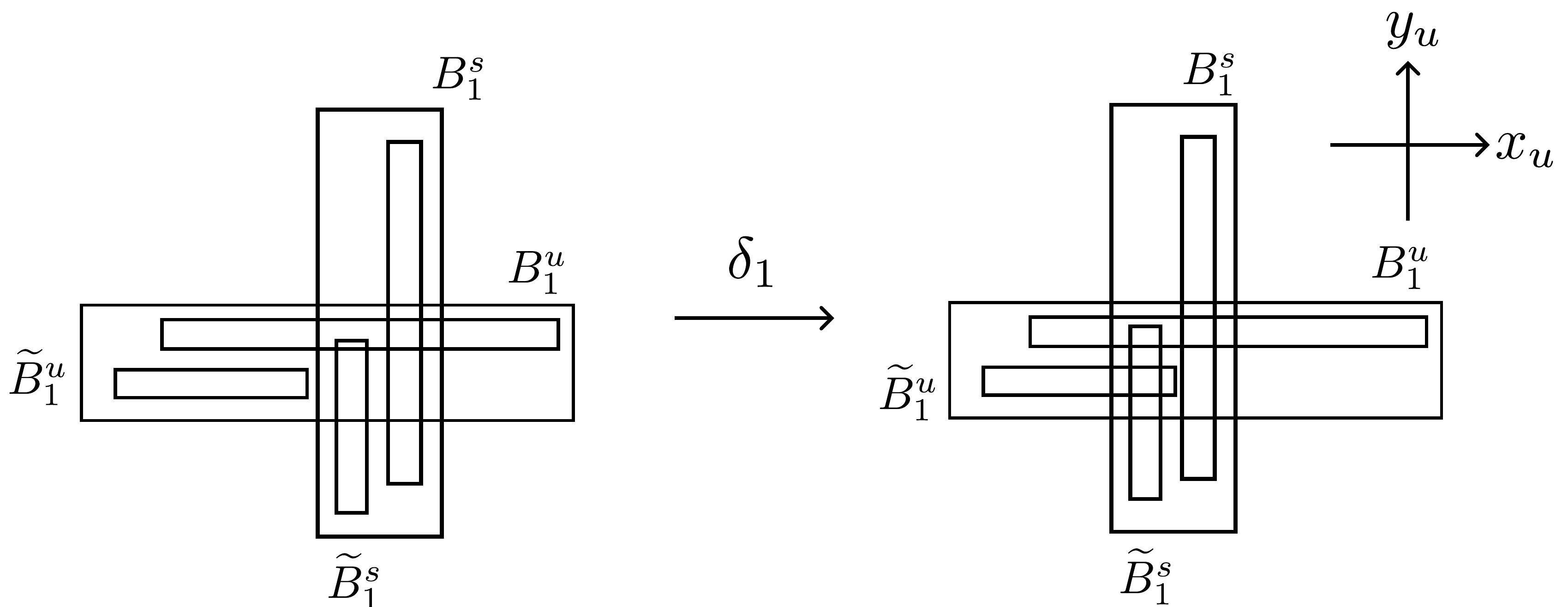}}
      \caption{The perturbed linked pairs $(\delta_{1} + \widetilde{B}^{s}_{1}, \widetilde{B}^{u}_{1})$ and $(\delta_{1} + B^{s}_{1},B^{u}_{1})$.}
      \label{us-box}
    \end{figure}

    Moreover, by \eqref{ll_1}, by taking the code length sufficiently large, we may choose the $u$-box $\bar{B}^{u}_{1}$ and the $s$-box $\bar{B}^{s}_{1}$ so that
    \begin{align*}
        |P(B^{s}_{1})| < |P(B^{u}_{1})|, \quad |P(\widetilde{B}^{s}_{1})| < |P(\widetilde{B}^{u}_{1})|,\\
        |Q(B^{u}_{1})|  < |Q(B^{s}_{1})|, \quad |Q(\widetilde{B}^{u}_{1})| < |Q(\widetilde{B}^{s}_{1})|.
    \end{align*}
    Hence there exists $\delta_{1} \in \mathbb{R}^{2}$ satisfying
    \begin{equation}\label{fir_lin}
        \begin{split}
            P(\delta_{1} +\widetilde{B}^{s}_{1}) \subset P(\widetilde{B}^{u}_{1}), \quad P(\delta_{1} + B^{s}_{1}) \subset P(B^{u}_{1}), \\
            Q(\widetilde{B}^{u}_{1}) \subset Q(\delta_{1} + \widetilde{B}^{s}_{1}), \quad Q(B^{u}_{1}) \subset Q(\delta_{1} + B^{s}_{1}).
        \end{split}
    \end{equation}
    Such a vector $\delta_{1}$ satisfying these inclusion relations belongs to $\bar{I}_{1}$.
    See also Figure \ref{us-box}.
    Moreover, by \eqref{ll_1} and by the definitions of $\widetilde{B}^{s}_{1}$ and $B^{s}_{1}$, we obtain
    \begin{equation*}
        \frac{\lambda_{cs0}^{3}}{2}\epsilon \leq |\widetilde{B}^{s}_{1}| < \frac{\lambda_{cs0}^{2}\epsilon}{2}, \quad\frac{\lambda_{cs0}^{2}\lambda_{cs1}}{2}\epsilon \leq |B^{s}_{1}| < \frac{\lambda_{cs0}\lambda_{cs1}}{2}\epsilon.
    \end{equation*}
    The $s$-box $\widetilde{B}^{s}_{1}$ and the $u$-box $\widetilde{B}^{u}_{1}$ will be used to construct the next $s$-box and $u$-box.
    From this point on, we prepare for an inductive construction of $u$-boxes and $s$-boxes.
    In order to apply the blender property in this procedure, we restrict our argument to appropriate subsets of the corresponding $u$-boxes and $s$-boxes.
    More precisely, for $\widetilde{B}^{s}_{1}= B^{s}(1^{m_{1}-1}0)$ and $B^{s}_{1} = B^{s}(1^{m_{1}})$, we consider the subsets
    \begin{equation*}
        \widetilde{R}^{s}_{1} = R^{s}(1^{m_{1}-1}0), \ R^{s}_{1} = R^{s}(1^{m_{1}}).
    \end{equation*}
    In what follows, we will use these restricted subsets of $u$-boxes and $s$-boxes only when invoking the blender property.
    We observe that $Q(\widetilde{R}^{s}_{1} \cap R^{s}_{1}) \neq \emptyset$, and the relative proportion of $Q(\widetilde{R}^{s}_{1} \cap R^{s}_{1})$ inside $\widetilde{R}^{s}_{1}$ satisfies
    \begin{equation*}
        \frac{|Q(\widetilde{R}^{s}_{1} \cap R^{s}_{1})|}{|\widetilde{R}^{s}_{1}|} > \frac{-1 + \lambda_{cs1}a_{s}}{3} = \xi_{0}.
    \end{equation*}
    Since $\bar{B}^{s}_{1}$ and $\bar{B}^{u}_{1}$ form a linked pair, and by the choice of the constants $\mu_{*}$ and $\lambda_{*}$ in \eqref{star}, we can choose $\delta_{1} \in \bar{I}_{1}$ so that
    \begin{equation}\label{lin_per}
        P(\delta_{1} + \bar{B}^{s}_{1}) \subset P(\widetilde{R}^{u}_{1} \cap R^{u}_{1}), \ Q(\bar{B}^{u}_{1}) \subset Q(\delta_{1} + (\widetilde{R}^{s}_{1} \cap R^{s}_{1})).
    \end{equation}
    Even under this perturbation, \eqref{fir_lin} still holds.
    \begin{claim}\label{claim}
        There exists a linked pair consisting of a $u$-box $\bar{B}^{u}_{2} \subset \widetilde{B}^{u}_{1}$ and an $s$-box $\bar{B}^{s}_{2} \subset \widetilde{B}^{s}_{1}$ such that the following inequality holds:
        \begin{equation}\label{lll}
            \frac{\lambda_{cs0}^{2}\xi_{0}}{8}|\widetilde{B}^{s}_{1}| \leq |\bar{B}^{s}_{2}| \leq \frac{\lambda_{cs0}\xi_{0}}{8}|\widetilde{B}^{s}_{1}|, \ \lambda_{cu0}^{-1} |\bar{B}^{u}_{2}| < |\bar{B}^{s}_{2}| \leq |\bar{B}^{u}_{2}|.
        \end{equation}
    \end{claim}
    \begin{proof}
        By the definition of a linked pair, $\widetilde{B}^{u}_{1}$ and $\widetilde{B}^{s}_{1}$ satisfy
        \begin{equation*}
            \widetilde{B}^{u}_{1} \cap \widetilde{B}^{s}_{1} = P(\widetilde{B}^{s}_{1} ) \times Q(\widetilde{B}^{u}_{1})
        \end{equation*}
        with $n_{1}$ and $m_{1}$ being the lengths of the codes of $\widetilde{B}^{u}_{1}$ and $\widetilde{B}^{s}_{1}$, respectively.
        Therefore, there exist intervals $J \subset P(Z)$ and $J^{\prime} \subset Q(\widetilde{Z})$ such that
        \begin{align*}
            f^{n_1}(\widetilde{B}^{u}_{1} \cap \widetilde{B}^{s}_{1}) &= J \times [-2,2], \\
            g^{m_1} \circ f^{-1}_{*} (\widetilde{B}^{u}_{1} \cap \widetilde{B}^{s}_{1}) &= [-2,2] \times J^{\prime}.
        \end{align*}

        \begin{figure}[H]
            \centering
            \scalebox{0.17}{\includegraphics{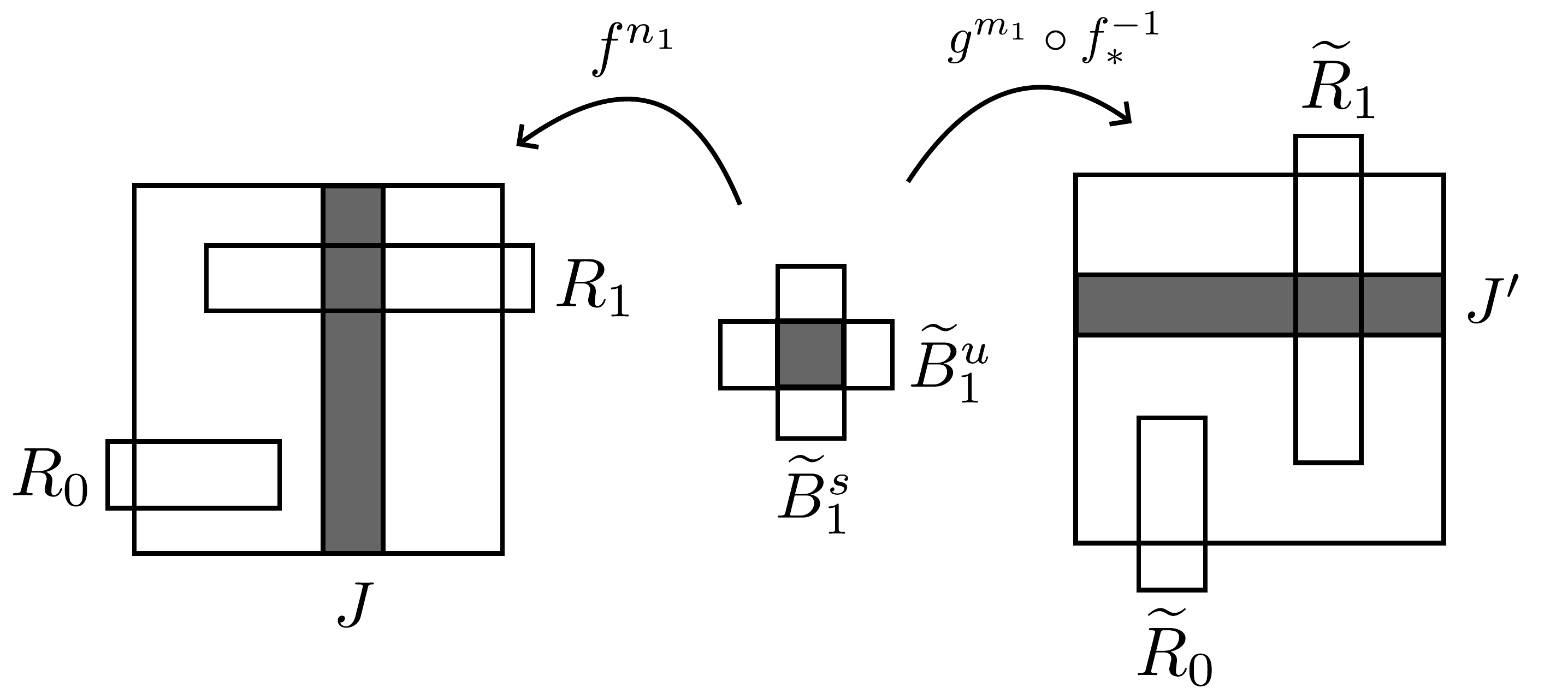}}
            \caption{}
            \label{cl}
        \end{figure}

        Here, the length of the interval $J$ is given by
        \begin{equation*}
            |J| = 4\lambda_{s}^{m_1}\lambda_{cu0}\lambda_{cu1}^{n_{1}-1}\lambda_{*}^{-1}.
        \end{equation*}
        By \eqref{ll_1}, since $\lambda_{s}^{m_{1}}\lambda_{cu0}\lambda_{cu1}^{n_{1}-1} <1$ the length of the interval $J$ can be made smaller than the Lebesgue number of the open covering $(\text{Int}\, (P(R_{i})))_{i=0,1}$ of $P(Z)$ by an appropriate choice of the constant $\lambda_{*}^{-1}$.
        Consequently, there exists $i \in \{0,1\}$ such that
        \begin{equation*}
            J \subset P(R_{i}).
        \end{equation*}
        On the other hand, the length of the interval $J^{\prime}$ satisfies
        \begin{equation*}
            |J^{\prime}| = 4\lambda_{u}^{-n_1}\lambda_{cs0}^{-1}\lambda_{cs1}^{-(m_{1}-1)}\mu_{*}^{-1}.
        \end{equation*}
        By the choice of the constants $\lambda_{u}$ and $\lambda_{cs0}$, we have $\lambda_{u}\lambda_{cs0} <1$.
        Therefore, in the process of constructing a linked pair by extending the codes, the increase in the number of symbols $0$ may cause the length of $J^{\prime}$ to exceed the Lebesgue number of the open covering $(\text{Int}\, (Q(\widetilde{R}_{i})))_{i=0,1}$ of $Q(\widetilde{Z})$.
        To avoid this issue, when constructing a linked pair with respect to the symbol $0$, we increase the code of the $u$-box by two symbols. That is,
        \begin{align*}
            |J^{\prime}_{new}| &= 4\lambda_{u}^{-(n_{1}+2)}\lambda_{cs0}^{-2}\lambda_{cs1}^{-(m_{1}-1)}\mu_{*}^{-1} \\
            &= (\lambda_{u}^{-2}\lambda_{cs0}^{-1})|J^{\prime}|.
        \end{align*}
        By \eqref{claim_con}, we have $\lambda_{u}^{-2}\lambda_{cs0}^{-1}<1$, the length of the interval $J^{\prime}_{new}$ is smaller than the Lebesgue number of the open covering $(\text{Int}\, (Q(\widetilde{R}_{i})))_{i=0,1}$.
        Hence, there exists $j \in \{0,1\}$ such that
        \begin{equation*}
            J^{\prime}_{new} \subset Q(\widetilde{R}_{j}).
        \end{equation*}
        Therefore, there exist $J_{1} \subset P(Z)$ and $J^{\prime}_{1} \subset Q(\widetilde{Z})$ such that
        \begin{align*}
            f^{n_{1}+1}(B^{u}(\underline{a}^{(1)}i) \cap B^{s}(\underline{w}^{(1)}j)) &= J_{1} \times [-2,2],\\
            g^{m_{1}+1}(B^{u}(\underline{a}^{(1)}i) \cap B^{s}(\underline{w}^{(1)}j)) &= [-2,2] \times J^{\prime}_{1}
        \end{align*}
        or
        \begin{align*}
            f^{n_{1}+2}(B^{u}(\underline{a}^{(1)}i_{1}i_{2}) \cap B^{s}(\underline{w}^{(1)}j)) &= J_{1} \times [-2,2],\\
            g^{m_{1}+1}(B^{u}(\underline{a}^{(1)}i_{1}i_{2}) \cap B^{s}(\underline{w}^{(1)}j)) &= [-2,2] \times J^{\prime}_{1}.
        \end{align*}
        Here $i,j \in \{0,1\}$ in the first case, while $i_1,i_2,j \in \{0,1\}$ in the second case.
        Hence, the $u$-box corresponding to $\underline{a}^{(1)}i$ (respectively, $\underline{a}^{(1)}i_1 i_2$) and the $s$-box corresponding to $\underline{w}^{(1)}j$ form a linked pair.
        By this procedure, we can shrink each box so as to satisfy the required inequalities while preserving the linked pair.
    \end{proof}
    As in the case $k=1$, we define
    \begin{equation*}
        \bar{I}_{2} = \qty{ \delta \in \mathbb{R}^{2} \mid (\delta + \bar{B}^{s}_{2}) \cap \bar{B}^{u}_{2} \neq \emptyset }.
    \end{equation*}
    Then, the length of the sides of $\bar{I}_{2}$ are given by
    \begin{equation*}
        |P(\bar{I}_{2})| = |\bar{B}^{u}_{2}| + |P(\bar{B}^{s}_{2})|, \quad |Q(\bar{I}_{2})| = |\bar{B}^{s}_{2}| + |Q(\bar{B}^{u}_{2})|.
    \end{equation*}
    By \eqref{lll}, the length of sides of $\bar{I}_{2}$ satisfy the following inequalities:
    \begin{align*}
        |P(\bar{I}_{2})|, |Q(\bar{I}_{2})| < |\bar{B}^{s}_{2}| + |\bar{B}^{u}_{2}| &< \frac{\xi_{0}}{4}|\widetilde{B}^{s}_{1}| \\
        &< \frac{\lambda_{cs0}^{2}\xi_{0}}{8}\frac{\epsilon}{2}.
    \end{align*}
    In particular, since $\bar{B}^{s}_{2}$ and $\bar{B}^{u}_{2}$ form a linked pair, we have $(0,0) \in \bar{I}_{2}$.
    Therefore, any $\delta_{2} \in \bar{I}_{2}$ satisfies
    \begin{equation*}
        |\delta_{2}| < \frac{\lambda_{cs0}^{2}\xi_{0}}{8}\frac{\epsilon}{2}
    \end{equation*}
    We take the related pairs of $\bar{B}^{s}_{2}$ and $\bar{B}^{u}_{2}$, denoted $\widetilde{B}^{s}_{2}$, $B^{s}_{2}$ and $\widetilde{B}^{u}_{2}$, $B^{u}_{2}$, respectively.
    Then there exists $\delta_{2} \in \mathbb{R}^{2}$ satisfying
    \begin{align*}
        P(\delta_{2} +\widetilde{B}^{s}_{2}) \subset P(\widetilde{B}^{u}_{2}), \quad P(\delta_{2} + B^{s}_{2}) \subset P(B^{u}_{2}), \\
        Q(\widetilde{B}^{u}_{2}) \subset Q(\delta_{2} + \widetilde{B}^{s}_{2}), \quad Q(B^{u}_{2}) \subset Q(\delta_{2} + B^{s}_{2}).
    \end{align*}
    Such a vector $\delta_{2}$ satisfying the inclusions belongs to $\bar{I}_{2}$.
    We then set $\Delta_{2} = \Delta_{1} + \delta_{2}$.

    Assume that for some $k \geq 1$, the following inductive hypotheses (i)--(iv) hold:
    \begin{itemize}
        \item [(i)] We choose a perturbation vector $\delta_{t} \in \bar{I}_{t}$ where
        \begin{equation*}
            \bar{I}_t := \{ \delta \mid (\delta + \bar{B}^{s}_{t}) \cap \bar{B
            }^{u}_{t} \neq  \emptyset \}.
        \end{equation*}
        Moreover, the vector $\delta_{t}$ satisfies
        \begin{equation*}
            |\delta_{t}| < \qty(\frac{\lambda_{cs0}^{2}\xi_{0}}{8})^{t-1}\frac{\epsilon}{2}
        \end{equation*}
        for all $t = 1 \dots k$.
        \item [(ii)] For each $t = 1, \dots k$, the pair $(\Delta_{k} + B^{s}_{t}, B^{u}_{t})$ is a linked $\lambda_{cu0}^{-1}$-proportional pair.
        \item [(iii)] The $s$-box $B^{s}_{t}$ satisfies
        \begin{equation*}
            \frac{\lambda_{cs0}^{3}}{8}|B^{s}_{t-1}| \leq |B^{s}_{t}| \leq \frac{\lambda_{cs0}^{2}}{8}|B^{s}_{t-1}|.
        \end{equation*}
        for all $t = 1, \dots k$.
        \item [(iv)] Let $\widetilde{B}^{u}_{k}$ be the related $u$-box of $B^{u}_{k}$, and let $\widetilde{B}^{s}_{k}$ be the related $s$-box of $B^{s}_{k}$.
        Then, under the perturbation $\Delta_{k} = \delta_{1} + \delta_{2} + \cdots + \delta_{k}$, the pair $(\Delta_{k} + \widetilde{B}^{s}_{k}, \widetilde{B}^{u}_{k})$ and $(\Delta_{k} + B^{s}_{k}, B^{u}_{k})$ are linked $\lambda_{cu0}^{-1}$-proportional pairs.
    \end{itemize}
    We will show that the same properties also hold for $k+1$.
    Since $(\Delta_{k} + \widetilde{B}^{s}_{k}, \widetilde{B}^{u}_{k})$ is a linked pair.
    By claim \ref{claim}, there exists a linked pair consisting of a $u$-box $\bar{B}^{u}_{k+1}$ and an $s$-box $\bar{B}^{s}_{k+1}$ satisfying the following the inequality.
    \begin{equation}\label{ll_2}
        \frac{\lambda_{cs0}^{2}\xi_{0}}{8}|\widetilde{B}^{s}_{k}| \leq |\bar{B}^{s}_{k+1}| \leq \frac{\lambda_{cs0}\xi_{0}}{8}|\widetilde{B}^{s}_{k}| , \quad
        \lambda_{cu0}^{-1}|\bar{B}^{u}_{k+1}| < |\bar{B}^{s}_{k+1}| \leq |\bar{B}^{u}_{k+1}|.
    \end{equation}
    Let $\bar{I}_{k+1}$
    \begin{equation*}
        \bar{I}_{k+1} = \qty{\delta \in \mathbb{R}^{2} \mid (\delta + \bar{B}^{s}_{k+1}) \cap \bar{B}^{u}_{k+1} \neq \emptyset }.
    \end{equation*}
    Each side length of $\bar{I}_{k+1}$ is given by
    \begin{equation*}
        |P(\bar{I}_{k+1})| = |\bar{B}^{u}_{k+1}| + |P(\bar{B}^{s}_{k+1})|, \ |Q(\bar{I}_{k+1})| = |\bar{B}^{s}_{k+1}| + |Q(\bar{B}^{u}_{k+1})|.
    \end{equation*}
    By \eqref{ll_2}, the side lengths of $\bar{I}_{k+1}$ satisfy
    \begin{align*}
        |P(\bar{I}_{k+1})|, |Q(\bar{I}_{k+1})| &< |\bar{B}^{u}_{k+1}| + |\bar{B}^{s}_{k+1}|\\
        &< \frac{\xi_{0}}{4}|\widetilde{B}^{s}_{k}|\\
        &<  \qty(\frac{\lambda_{cs0}^{2}\xi_{0}}{8})^{k}\frac{\epsilon}{2}.
    \end{align*}
    Since $\bar{B}^{s}_{k+1}$ and $\bar{B}^{u}_{k+1}$ form a linked pair, we have $(0,0) \in \bar{I}_{k+1}$. The vector $\delta_{k+1} \in \bar{I}_{k+1}$ satisfies
    \begin{equation*}
        |\delta_{k+1}| < \qty(\frac{\lambda_{cs0}^{2}\xi_{0}}{8})^{k}\frac{\epsilon}{2}.
    \end{equation*}
    Under the perturbation $\Delta_{k+1} = \Delta_{k} + \delta_{k+1}$, the pairs $(\Delta_{k+1} +\widetilde{B}^{s}_{k+1}, \widetilde{B}^{u}_{k+1})$ and $(\Delta_{k+1} +B^{s}_{k+1}, B^{u}_{k+1})$ become linked $\lambda_{cu0}^{-1}$-proportional pairs.
    By the inductive hypothesis, for each $t = 1, \dots k$, the pair $(\Delta_{t} + B^{s}_{t}, B^{u}_{t})$ forms a linked pair.
    Moreover, by the construction of the $u$-boxes and $s$-boxes inside the intersections $\widetilde{R}^{s}_{t} \cap R^{s}_{t}$ and $\widetilde{R}^{u}_{t} \cap R^{u}_{t}$ as in \eqref{lin_per}, the pairs $(\Delta_{t} + B^{s}_{t}, B^{u}_{t})$ admit perturbations of size at least $|\bar{B}^{s}_{t}|/2$.
    By condition (i), we obtain the following inequality.
    \begin{align*}
        |\Delta_{k+1} - \Delta_{t}| \leq \sum_{j=t+1}^{k+1}|\delta_{j}|
        &< |\bar{B}^{s}_{t}|\qty(\frac{\lambda_{cs0}\xi_{0}}{4} \cdot \frac{\lambda_{cs0}^{2}\xi_{0}}{8} + \dots \frac{\lambda_{cs0}\xi_{0}}{4} \qty(\frac{\lambda_{cs0}^{2}\xi_{0}}{8})^{k-t+1}), \\
        &< \frac{|\bar{B}^{s}_{t}|}{2}.
    \end{align*}
    Since $\Delta_{k+1} = \Delta_{t} + (\Delta_{k+1} - \Delta_{t})$ and $|\Delta_{k+1} - \Delta_{t}| < |\bar{B}^{s}_{t}|/2$, for each $t = 1, \dots k+1$, the pair $(\Delta_{k+1} + B^{s}_{t}, B^{u}_{t})$ forms a linked $\lambda_{cu0}^{-1}$-proportional pair.
    Finally, define
    \begin{equation*}
        \Delta = \lim_{k \to \infty}\Delta_{k}= \sum_{k=1}^{\infty} \delta_{k}.
    \end{equation*}
    Then
    \begin{equation*}
        |\Delta| < \sum_{k=1}^{\infty} |\delta_{k}| < \frac{\epsilon}{2}\sum_{k=1}^{\infty}\qty(\frac{\lambda_{cs0}^{2}\xi_{0}}{8})^{k-1} < \epsilon,
    \end{equation*}
    and hence (1) and (2) follows.

    For (3), condition (ii) in the construction of the sequence of $s$-boxes implies that
    \begin{equation}\label{lgl_1}
        |B^{s}_{k+1}| \geq \frac{\lambda_{cs0}^{3}\xi_{0}}{8}|B^{s}_{k}|.
    \end{equation}
    Since $\lambda_{s} < \lambda_{cs0} < \lambda_{cs1}<1$, inequality \eqref{lgl_1} yields
    \begin{equation*}
        \frac{\lambda_{cs0}^{3}\xi_{0}}{8} < \lambda_{cs1}^{|\underline{w}^{(k+1)}|-|\underline{w}^{(k)}|}.
    \end{equation*}
    Therefore,
    \begin{equation*}
        |\underline{w}^{(k+1)}| < |\underline{w}^{(k)}| + \frac{\log(\lambda_{cs0}^{3}\xi_{0}/8)^{-1}}{\log \lambda_{cs1}^{-1}}.
    \end{equation*}
    Let $N_{s}$ be the smallest integer satisfying
    \begin{equation*}
        N_{s} \geq \frac{\log(\lambda_{cs0}^{3}\xi_{0}/8)^{-1}}{\log \lambda_{cs1}^{-1}}.
    \end{equation*}
    Then the desired estimate follows.
    Furthermore, since the $s$-box $\bar{B}^{s}_{k+1}$ and the $u$-box $\bar{B}^{u}_{k+1}$ are proportional by \eqref{ll_2}, the size estimate for $\bar{B}^{s}_{k+1}$ in \eqref{ll_2} implies that
    \begin{equation*}
        |B^{u}_{k+1}| \geq |B^{s}_{k+1}| \geq \frac{\lambda_{cs0}^{3}\xi_{0}}{8}|B^{s}_{k}| > \frac{\lambda_{cs0}^{3}\lambda_{cu0}^{-1}\xi_{0}}{8}|B^{u}_{k}|.
    \end{equation*}
    Hence,
    \begin{equation*}
        |\underline{a}^{(k+1)}| < |\underline{a}^{(k)}| + \frac{\log (\lambda_{cs0}^{3}\lambda_{cu0}^{-1}\xi_{0}/8)^{-1}}{\log \lambda_{cu1}}.
    \end{equation*}
    Let $N_{u}$ be the smallest integer satisfying
    \begin{equation*}
        N_{u} \geq \frac{\log (\lambda_{cs0}^{3}\lambda_{cu0}^{-1}\xi_{0}/8)^{-1}}{\log \lambda_{cu1}}.
    \end{equation*}
    This complete the proof of the lemma.
\end{proof}

\begin{prop}\label{critical}
    For every $L \gg N_{s}$ there exist a sequence $\underline{\Delta}(L) = (\Delta_{k})_{k \geq 1}$ with $|\Delta_{1}|=0$ and $|\Delta_{k}|< \lambda_{s}^{Lk}$ such that the following hold.
    For every codes $\underline{u}^{(k)}$ $(k = 1,2,\dots)$, there exists a code $\underline{\gamma}^{(k)}$ $(k=1,2,\dots)$ such that for each $k\geq 1$, let $\widehat{B}^{s}_{k+1} \subset B^{s}_{k+1}$ and $\widehat{B}^{u}_{k} \subset B^{u}_{k}$ denote the $s$-box and $u$-box determined by the code $\underline{\gamma}^{(k)}$.
    Then the center $\Delta_{k+1} + \widehat{B}^{s}_{k+1}$ coincides with the center of $\widehat{B}^{u}_{k+1}$.
    Moreover, each code $\underline{\gamma}^{(k)}$ can be written in the form
    \begin{equation*}
        \underline{\gamma}^{(k)} = \underline{\alpha}^{(k)} \underline{u}^{(k)} [\underline{\omega}^{(k+1)}]^{-1}
    \end{equation*}
    where $\underline{\alpha}^{(k)}$ and $\underline{\omega}^{(k+1)}$ satisfy
    \begin{equation*}
        |\underline{\alpha}^{(k)}| + |\underline{\omega}^{(k+1)}| \leq CLk,
    \end{equation*}
    with a constant $C>0$ that is independent of $k$.
\end{prop}
\begin{proof}
    By claim \ref{claim} in the proof of lemma \ref{ll}, for each $k\geq 1$ there exist an $s$-box $\check{B}^{s}_{k} \subset B^{s}_{k}$ and a $u$-box $\check{B}^{u}_{k} \subset B^{u}_{k}$ such that the pair $(\check{B}^{s}_{k}, \check{B}^{u}_{k})$ is a linked pair and satisfies
    \begin{equation}\label{cc1}
        \frac{\lambda_{cs0}^{2}}{2}\lambda_{s}^{Lk} \leq |\check{B}^{s}_{k}| < \frac{\lambda_{cs0}}{2}\lambda_{s}^{Lk}, \quad \lambda_{cu0}^{-1}|\check{B}^{u}_{k}| < |\check{B}^{s}_{k}| \leq |\check{B}^{u}_{k}|.
    \end{equation}
    Define
    \begin{equation*}
        \check{I}_{k} = \qty{ \Delta_{k} \in \mathbb{R}^{2} \mid \qty(\Delta_{k} + \check{B}^{s}_{k}) \cap \check{B}^{u}_{k} \neq \emptyset}.
    \end{equation*}
    By the same argument as in lemma \eqref{ll}, the side length of $\check{I}_{k}$ satisfy
    \begin{equation*}
        |P(\check{I}_{k})|, |Q(\check{I}_{k})| < |\check{B}^{s}_{k}| + |\check{B}^{u}_{k}| < \lambda_{s}^{Lk}
    \end{equation*}
    Since $(\check{B}^{s}_{k}, \check{B}^{u}_{k})$ is a linked pair, the size estimates in \eqref{cc1} imply that the vector $\Delta_{k} \in \check{I}_{k}$ is bounded above by $\lambda_{s}^{Lk}$.
    Let $\underline{\alpha}^{(k)}$ and $\underline{\omega}^{(k)}$ denote the codes of $\check{B}^{u}_{k}$ and $\check{B}^{s}_{k}$, respectively.
    For an arbitrary code $\underline{u}^{(k)}$, define
    \begin{equation*}
        \underline{\gamma}^{(k)} = \underline{\alpha}^{(k)} \underline{u}^{(k)} [\underline{\omega}^{(k+1)}]^{-1}.
    \end{equation*}
    Then, by the definition of the $u$- and $s$-box, we obtain
    \begin{align*}
        \widehat{B}^{u}_{k} = B^{u}(\underline{\gamma}^{(k)}) \subset \check{B}^{u}_{k}, \quad \widehat{B}^{s}_{k+1} = B^{s}([\underline{\gamma}^{(k)}]^{-1}) \subset \check{B}^{s}_{k+1}.
    \end{align*}
    Finally, by choosing $\Delta_{k+1}$ so that the centers of the $u$-box $\widehat{B}^{u}_{k+1}$ and the $s$-box $\widehat{B}^{s}_{k+1}$ coincide, we obtain the desired property.
    Moreover, using \eqref{cc1} together with the relations among the constants in \eqref{2_con1}, the same computation as in lemma \ref{ll}, we obtain
    \begin{align*}
        |\underline{\alpha}^{(k)}| &< \frac{\log(\lambda_{cs0}^{2}/2)^{-1}}{\log \lambda_{cu1}} + \frac{\log \lambda_{s}^{-1}}{\log \lambda_{cu1}}Lk, \\
        |\underline{\omega}^{(k+1)}| &< \frac{\log(\lambda_{cs0}^{2}/2)^{-1}}{\log \lambda_{cs1}^{-1}} + \frac{\log \lambda_{s}^{-1}}{\log \lambda_{cs1}^{-1}}Lk + \frac{\log \lambda_{s}^{-1}}{\log \lambda_{cs1}^{-1}}.
    \end{align*}
    Since $\lambda_{cu1}\lambda_{cs1} < 1$, if we take
    \begin{equation*}
        C > \max \qty{\frac{\log(\lambda_{cs0}^{2}/2)^{-1}}{\log \lambda_{cu1}}, \frac{\log \lambda_{s}^{-1}}{\log \lambda_{cu1}}}.
    \end{equation*}
    large enough, then
    \begin{equation*}
        |\underline{\alpha}^{(k)}| + |\underline{\omega}^{(k+1)}| < CLk.
    \end{equation*}
    holds for every $k$.
    This completes the proof of the proposition.
\end{proof}

Theorem \ref{main2} follows from lemma \ref{ll} and proposition \ref{critical}.

\section{Proof of Main Theorem}\label{proof}
In this section, we describe the proof of theorem \ref{main} by using the codes $\underline{\gamma}^{(k)}$ constructed in theorem \ref{main2}.

\subsection{Perturbation}\label{per}
We now apply sufficiently small perturbations near homoclinic tangencies, we construct a diffeomorphism that is $C^{r}$-close to $F$ and exhibits the statistical behavior in theorem \ref{main}.
Here, we denote by $\mathbb{V}_{*}^{c}$ the complement of $\mathbb{V}_{*}$ defined in \eqref{tan_dom}.

\begin{prop}
    For any $\epsilon>0$, there exists a diffeomorphism $G$ contained in the $\epsilon$-neighborhood $\mathcal{U}$ of $F$ in the $C^{r}$-topology $(1 \leq r < \infty)$ such that the following properties hold:
    \begin{itemize}
        \item [\rm{(1)}] The restriction of $G$ to $\mathbb{V}_{*}^{c}$ coincides with $F$.
        \item [\rm{(2)}] Fix a constant $0 < \delta < 1-2\lambda_{u}^{-1}$, and let $\Delta_{k} = (t_{k}, \widetilde{t}_{k}, 0, 0)$ be the vector obtained in proposition \ref{critical}.
        Then, for every $(x_{u}, y_{u}, x_{s}, y_{s}) \in [ -\delta,  \delta]^{2} \times B^{s}_{k+1}$, we have
        \begin{equation*}
            G(x_{u}, y_{u}, x_{s},y_{s}) = \Delta_{k+1} + F(x_{u},y_{u},x_{s},y_{s}).
        \end{equation*}
    \end{itemize}
\end{prop}
\begin{proof}
    Let $\boldsymbol{b} \colon \mathbb{R} \rightarrow \mathbb{R}$ be a non-negative, non-decreasing $C^{r}$-function such that
    \begin{equation*}
        \boldsymbol{b}(x) =
        \begin{cases}
            0 & x \leq -1 ,\\
            1 & x \geq 0.
        \end{cases}
    \end{equation*}
    For an interval $I=[a,b]$, we define the bump function $\boldsymbol{b}_{\rho,I}$ by
    \begin{equation*}
        \boldsymbol{b}_{\rho,I} = \boldsymbol{b}\qty(\frac{x-a}{\rho|I|}) + \boldsymbol{b}\qty(-\frac{x-b}{\rho|I|}) -1
    \end{equation*}
    where $|I| = b-a$ denotes the length of $I$.
    If $\rho|I| \leq 1$, then
    \begin{equation}\label{per_1}
        \|\boldsymbol{b}_{\rho,I}\|_{C^{r}} \leq \frac{1}{(\rho|I|)^{r}} \|\boldsymbol{b}\|_{C^{r}}.
    \end{equation}
    Here, $\|\cdot\|_{C^{r}}$ denotes the supremum norm of derivatives up to order $r$.
    We next define bump functions corresponding to each variables as follows:
    \begin{equation}\label{bump}
        b_{cu} = \boldsymbol{b}_{1/4,[-\delta,\delta]},\ b_{u} = \boldsymbol{b}_{1/4,[-\delta,\delta]},\ b_{s,k} = \boldsymbol{b}_{1/3\tau^{s},P(B^{s}_{k})},
    \end{equation}
    where $\tau^{s} = \tfrac{2\lambda_{s}}{1-2\lambda_{s}}$.
    For the functions in \eqref{bump}, the inequality \eqref{per_1} provides the following estimates:
     \begin{align*}
        \|b_{cu}\|_{C^{r}} \leq \qty(\frac{2}{\delta})^{r} \|\boldsymbol{b}\|_{C^{r}}, \ \|b_{u}\|_{C^{r}} \leq \qty(\frac{2}{\delta})^{r} \|\boldsymbol{b}\|_{C^{r}} ,\ \|b_{s,k}\|_{C^{r}} \leq \qty(\frac{3\tau^{s}}{|P(B^{s}_{k})|})^{r}\|\boldsymbol{b}\|_{C^{r}}.
    \end{align*}
    For each $k \geq 1$, the vector $\Delta_{k+1}$ depends on $L>0$, and its norm is bounded from above by proposition \ref{critical}.
    We set
    \begin{equation*}
        \underline{\Delta}(L) = (\Delta_{k})_{k \geq 1}.
    \end{equation*}
    We then define a map $h_{\underline{\Delta}(L)} \colon \mathbb{R}^{4} \rightarrow \mathbb{R}^{4}$ by
    \begin{align*}
        h_{\underline{\Delta}(L)} (x_{u},y_{u}, x_{s},y_{s}) = \biggl(&x_{u},\  y_{u}, \ x_{s} + \lambda_{*}b_{cu}(x_{u}) b_{u}(y_{u}) \sum_{k=1}^{\infty} t_{k+1} b_{s,k+1}(x_{s}),\\
        &y_{s} + \mu_{*}^{-1} b_{cu}(x_{u}) b_{u}(y_{u}) \sum_{k=1}^{\infty} \widetilde{t}_{k+1} b_{s,k+1}(x_{s})\biggr)
    \end{align*}
    for $(x_{u},y_{u}, x_{s},,y_{s}) \in [-\delta,  \delta]^{2} \times \cup_{k\geq 1}B^{s}_{k+1}$.
    Let $id \colon \mathbb{R}^{4} \to \mathbb{R}^{4}$ denote the identity map. Then we estimate
  \begin{align*}
        &\|h_{\underline{\Delta}(L)} - id \|_{C^{r}} \\
        &=\bigg\| \lambda_{*} b_{cu}(x_{u}) b_{u}(y_{u}) \sum_{k=1}^{\infty} t_{k+1} b_{s,k+1}(x_{s}) + \mu_{*}^{-1}b_{cu}(x_{u}) b_{u}(y_{u}) \sum_{k=1}^{\infty} \widetilde{t}_{k+1} b_{s,k+1}(x_{s})\bigg\|_{C^{r}} \\
        &\leq  \lambda_{*}\qty(\frac{12\tau^{s}}{\delta^{2}})^{r} \|\boldsymbol{b}\|_{C^{r}} \sum_{k=1}^{\infty} \frac{|t_{k+1}|}{|P(B^{s}_{k+1})|^{r}} +  \mu_{*}^{-1}\qty(\frac{12\tau^{s}}{\delta^{2}})^{r} \|\boldsymbol{b}\|_{C^{r}} \sum_{k=1}^{\infty} \frac{|\widetilde{t}_{k+1}|}{|P(B^{s}_{k+1})|^{r}}.
  \end{align*}
  By proposition \ref{critical}, we have $|t_{k+1}| < \lambda_{s}^{L(k+1)}$. Taking $L>0$ sufficiently large so that $\lambda_{s}^{L} < \lambda_{s}^{N_{s}r}$, it follows that
  \begin{equation}\label{4_per_t}
      \begin{split}
          \sum_{k=1}^{\infty}\frac{|t_{k+1}|}{|P(B^{s}_{k+1})|^{r}} &\leq \sum_{k=1}^{\infty} \frac{\lambda_{s}^{L(k+1)}}{(\lambda_{s}^{N_{s}k})^{r}} \\
          &= \frac{\lambda_{s}^{L}}{(\lambda_{s}^{N_{s}})^{r}} \sum_{k=1}^{\infty} \qty(\frac{\lambda_{s}^{L}}{\lambda_{s}^{N_{s}r}})^{k} \\
          &= \frac{\lambda_{s}^{L}}{(\lambda_{s}^{N_{s}})^{r}} \cdot \frac{\lambda_{s}^{L}}{\lambda_{s}^{N_{s}r}- \lambda_{s}^{L}}.
      \end{split}
  \end{equation}
  Similarly,
  \begin{align}\label{4_per_t2}
      \sum_{k=1}^{\infty} \frac{|\widetilde{t}_{k+1}|}{|P(B^{s}_{k+1})|^{r}} &\leq \frac{\lambda_{s}^{L}}{2( \lambda_{s}^{N_{s}})^{r}} \cdot \frac{\lambda_{s}^{L}}{\lambda_{s}^{N_{s}r} - \lambda_{s}^{L}}.
  \end{align}
  Hence, by taking $L>0$ sufficiently large, both \eqref{4_per_t} and \eqref{4_per_t2} converge to zero.
  Therefore $h_{\underline{\Delta}(L)}$ is $C^{r}$-close to the identity map $id$.
  Using the map $h_{\underline{\Delta}(L)}$ defined above, we set
  \begin{equation}\label{per_g}
      G = F \circ h_{\underline{\Delta}(L)}.
  \end{equation}
  Then, for $L>0$ sufficiently large, the map $G$ is $C^{r}$-close to $F$.
  By construction, the restriction $h_{\underline{\Delta}(L)}$ to $\mathbb{V}_{*}^{c}$ coincided with the identity map $id$.
  Moreover, for any $(x_{u}, y_{u}, x_{s}, y_{s}) \in [-\delta,  \delta]^{2} \times B^{s}_{k+1}$,
  \begin{align*}
      G(x_{u}, y_{u}, x_{s}, y_{s}) &= F \circ h_{\underline{\Delta}(L)}(x_{u}, y_{u}, x_{s}, y_{s}) \\
      &= F(x_{u},y_{u}, x_{s} +  \lambda_{*}t_{k+1}, y_{s} + \mu_{*}^{-1}\widetilde{t}_{k+1}) \\
      &= (t_{k+1}, \widetilde{t}_{k+1}, 0, 0) + F(x_{u}, y_{u}, x_{s}, y_{s}).
  \end{align*}
  This completes the proof.
\end{proof}

\subsection{Code conditions for wandering domains}\label{con_wan}
To construct the wandering domain, we use the conditions introduced in \cite{CV2001, KNS2023}.
The codes obtained in proposition \ref{critical} have the form
\begin{equation*}
    \underline{\gamma}^{(k)} = \underline{\alpha}^{(k)} \underline{u}^{(k)} [\underline{\omega}^{(k+1)}]^{-1}.
\end{equation*}
We denote the length of the code $\underline{\gamma}^{(k)}$ by
\begin{equation*}
    \widehat{n}_{k} =|\underline{\gamma}^{(k)}|.
\end{equation*}
By the construction of the code $\underline{\gamma}^{(k)}$ and by Theorem \eqref{critical}we have
\begin{equation*}
  |\underline{\alpha}^{(k)}| + |\underline{\omega}^{(k+1)}| < CLk.
\end{equation*}
Moreover, the sequence $\underline{u}^{(k)}$ is the part that can be chosen freely.
We therefore suppose the following quadratic condition:
\begin{equation}\label{wd_qu}
  |\underline{u}^{(k)}| = k^{2}.
\end{equation}
Under this assumption,
\begin{equation*}
  \frac{\widehat{n}_{k+1}}{\widehat{n}_{k}} = \frac{(k+1)^{2} + \mathcal{O}(k+1)}{k^{2} + \mathcal{O}(k)} \rightarrow 1 \quad (k \rightarrow \infty).
\end{equation*}
Hence, the following lemma holds.

\begin{lem}\label{lem4_1}
    For any $\eta>0$, there exists $k_{0} >0$ such that for all $k \geq k_{0}$,
    \begin{equation}
        \widehat{n}_{k} < \widehat{n}_{k+1} < (1+\eta)\widehat{n}_{k}.
    \end{equation}
\end{lem}

Furthermore, in addition to the condition \eqref{wd_qu}, we need to supppose another condition.
Since the sequence $\underline{u}^{(k)}$ can be chosen freely, we may assume that the number of symbols $0$ in $\underline{\gamma}^{(k)}$, denoted by $\widehat{n}_{k(0)}$ is greater than the number of symbols $1$, denoted by $\widehat{n}_{k(1)}$.
That is, we may assume that
\begin{equation}\label{maj}
    \widehat{n}_{k(1)} \leq \widehat{n}_{k(0)}.
\end{equation}
We refer to this as the majority condition.

\subsection{Construction of wandering domains}\label{con_wan2}
In what follows, we construct contracting wandering domains.
Since the estimates in the $x_{u}x_{s}$-plane and in the $y_{u}y_{s}$-plane can be carried out in essentially the same manner, we consider their projections onto each plane separately and proceed with the proof.

For each $k \geq k_{0}$, define
\begin{equation*}
  b_{k} = \lambda_{cu0}^{-\sum_{i=0}^{\infty}\widehat{n}_{k+i(0)}/2^{i}}\lambda_{cu1}^{-\sum_{i=0}^{\infty}\widehat{n}_{k+i(1)}/2^{i}}, \quad \bar{b}_{k} = a_{1}^{-1}\lambda^{-\sum_{i=0}^{\infty}\widehat{n}_{k+i}/2^{i}}_{u}.
\end{equation*}
Then the following relations hold:
\begin{equation}\label{4_con1}
(\lambda_{cu0}^{\widehat{n}_{k(0)}}\lambda_{cu1}^{\widehat{n}_{k(1)}})^{2}b^{2}_{k} = b_{k+1}, \quad a_{1}\lambda^{2\widehat{n}_{k}}_{u}\bar{b}^{2}_{k} = \bar{b}_{k+1}.
\end{equation}
Moreover, define
\begin{equation}\label{4_con2}
  x^{*}_{k} = 20a_{2}b^{1/2}_{k}, \quad  y^{*}_{k} = 20a^{-1}_{2}\bar{b}^{1/2}_{k}.
\end{equation}
Let $(\widehat{x}^{u}_{k}, \widehat{y}^{u}_{k})$ denote the center of the $u$-box $\widehat{B}^{u}_{k}$.
We then introduce rectangles $X_{k}$ in the $x_{u}x_{s}$-plane and $Y_{k}$ in the $y_{u}y_{s}$-plane by
\begin{align*}
  X_{k} &= \qty[\widehat{x}^{u}_{k} - \frac{b_{k}}{2}, \widehat{x}^{u}_{k} + \frac{b_{k}}{2}] \times \qty[ - x^{*}_{k},  x^{*}_{k}], \\
  Y_{k} &= \qty[\widehat{y}^{u}_{k} - \frac{\bar{b}_{k}}{2}, \widehat{y}^{u}_{k} + \frac{\bar{b}_{k}}{2}] \times \qty[-y^{*}_{k},  y^{*}_{k}].
\end{align*}
We finally define the four-dimensional region $W_{k}$ by
\begin{equation*}
  W_{k} = X_{k} \times Y_{k}.
\end{equation*}

The following theorem is the main result established in this section.

\begin{thm}\label{main3}
    There exists an integer $K \geq k_{0}$ such that, for every $k \geq K$, the set $D_{k} = \mathrm{Int}\,(W_{k})$ is a wandering domain for $G$, and it satisfies
    \begin{equation}
        G^{\widehat{n}_{k}+2}(D_{k}) \subset D_{k+1}.
    \end{equation}
\end{thm}

To prove this theorem, we introduce the projections $\pi \colon \mathbb{R}^{4} \to \mathbb{R}^{2}$ onto the $x_{u}x_{s}$-plane and $\pi^{\prime} \colon \mathbb{R}^{4} \to \mathbb{R}^{2}$ onto the $y_{u}y_{s}$-plane defined by
\begin{equation}\label{pro}
    \pi(x_{u}, y_{u}, x_{s}, y_{s}) = (x_{u}, x_{s}), \quad \pi^{\prime}(x_{u}, y_{u}, x_{s}, y_{s}) = (y_{u}, y_{s}).
\end{equation}

\begin{lem}\label{con_psi_phi}
    For any $(\widehat{x}^{u}_{k}+x_{u}, \widehat{y}^{u}_{k}+y_{u}, x_{s}, y_{s}) \in \widehat{B}^{u}_{k} \times [-2,2]^{2}$ the projections of $G^{\widehat{n}_{k}+1}(\widehat{x}^{u}_{k}+x_{u}, \widehat{y}^{u}_{k}+y_{u}, x_{s}, y_{s})$ onto the $x_{u}x_{s}$-plane and $y_{u}y_{s}$-plane are given by the following formulas:
    \begin{align}\label{varphi}
        (1)\ &\pi \circ G^{\widehat{n}_{k}+1}(\widehat{x}^{u}_{k}+x_{u}, \widehat{y}^{u}_{k}+y_{u}, x_{s}, y_{s}) \notag \\
        &= \qty(\widehat{x}^{u}_{k+1},0) + \qty(-(\lambda_{cu0}^{\widehat{n}_{k(0)}}\lambda_{cu1}^{\widehat{n}_{k(1)}}x_{u})^{2} + \lambda_{*}^{-1}(\lambda_{s}^{\widehat{n}_{k}}x_{s}-a_{s})+1,   a_{2}\lambda_{cu0}^{\widehat{n}_{k(0)}}\lambda_{cu1}^{\widehat{n}_{k(1)}}x_{u}).\\
        (2)\ &\pi^{\prime} \circ G^{\widehat{n}_{k}+1}(\widehat{x}^{u}_{k}+x_{u}, \widehat{y}^{u}_{k}+y_{u}, x_{s}, y_{s}) \notag \\
        &= \qty(\widehat{y}^{u}_{k+1},0) + \qty(-a_{1}\lambda^{2\widehat{n}_{k}}_{u} y_{u}^{2} + \mu_{*}(\lambda_{cs0}^{\widehat{n}_{k(0)}}\lambda_{cs1}^{\widehat{n}_{k(1)}}y_{s}-1)+a_{u}, \lambda_{u}^{\widehat{n}_{k}}y_{u}).
    \end{align}
\end{lem}
\begin{proof}
    First, we define the functions $\xi_{i} \colon \mathbb{R} \rightarrow \mathbb{R} \ (i = 0,1)$ by
    \begin{equation*}
        \xi_{0}(x_{u}) = \lambda_{cu0}(x_{u} + 1 ), \quad \xi_{1}(x_{u}) =\lambda_{cu1}(x_{u} - 1)+1.
    \end{equation*}
    Then, for any $\widehat{x}_{u}, x_{u} \in \mathbb{R}$, we have
    \begin{equation*}
        \xi_{0}(\widehat{x}_{u} + x_{u}) = \xi_{0}(\widehat{x}_{u}) + \lambda_{cu0}x_{u}, \quad \xi_{1}(\widehat{x}_{u} + x_{u}) = \xi_{1}(\widehat{x}_{u}) + \lambda_{cu1}x_{u}.
    \end{equation*}
    Similarly, define $\iota_{i}\colon \mathbb{R} \rightarrow \mathbb{R} \ (i = 0,1)$ by
    \begin{equation*}
        \iota_{0}(x_{s}) = \lambda_{s}x_{s}-1, \quad \iota_{1}(x_{s}) =\lambda_{s}x_{s} + 1.
    \end{equation*}
    Then, for any $\widehat{x}_{s},x_{s} \in \mathbb{R}$,
    \begin{equation*}
        \iota_{0} (\widehat{x}_{s} + x_{s}) = \iota_{0}(\widehat{x}_{s}) + \lambda_{s}x_{s}, \ \iota_{1}(\widehat{x}_{s}+ x_{s}) = \iota_{1}(\widehat{x}_{s}) + \lambda_{s}x_{s}.
    \end{equation*}
    Let $\underline{\gamma}^{(k)} = \gamma(1)\gamma(2) \ldots \gamma(\widehat{n}_{k})$ be the code associated with the boxes $\widehat{B}^{u}_{k}$ and $\widehat{B}^{s}_{k}$. Then, for each $i \in \{1, \ldots \widehat{n}_{k}\}$, we have
    \begin{align}\label{psi_2}
        &\pi \circ G^{\widehat{n}_{k}}(\widehat{x}^{u}_{k}+x_{u}, \widehat{y}^{u}_{k}+y_{u}, 1/2+x_{s}, 1/2+y_{s}) \notag \\
        & = \Big(\xi_{\gamma(\widehat{n}_{k})} \circ \cdots \circ \xi_{\gamma(1)}(\widehat{x}^{u}_{k}) + \lambda_{cu0}^{\widehat{n}_{k(0)}}\lambda_{cu1}^{\widehat{n}_{k(1)}}x_{u}, \iota_{\gamma(\widehat{n}_{k})} \circ \cdots \circ \iota_{\gamma_{1}}(0) + \lambda^{\widehat{n}_{k}}_{s}x_{s}\Big).
    \end{align}
    Since $\widehat{x}^{u}_{k}$ is the center of the $u$-box $\widehat{B}^{u}_{k}$, it follows from the definition of the $u$-box that
    \begin{equation*}
        \xi_{\gamma(\widehat{n}_{k})} \circ \cdots \circ \xi_{\gamma(1)}(\widehat{x}^{u}_{k}) = 0.
    \end{equation*}
    Moreover, if $\widehat{x}^{s}_{k+1}$ denotes the center of the $s$-box $\widehat{B}^{s}_{k+1}$, then
    \begin{equation*}
        \iota_{\gamma(\widehat{n}_{k})} \circ \cdots \circ \iota_{\gamma_{1}}(1/2) = \widehat{x}^{s}_{k+1}.
    \end{equation*}
    Hence, by \eqref{psi_2}
    \begin{align*}
        &\pi \circ G^{\widehat{n}_{k}}(\widehat{x}^{u}_{k}+x_{u}, \widehat{y}^{u}_{k}+y_{u}, x_{s}, y_{s}) \\
        &= \qty( \lambda_{cu0}^{\widehat{n}_{k(0)}} \lambda_{cu1}^{\widehat{n}_{k(1)}}x_{u},
    \widehat{x}^{s}_{k+1} + \lambda^{\widehat{n}_{k}}_{s}x_{s}) \in [-\delta, \delta] \times [-2,2].
    \end{align*}
    Therefore,
    \begin{align*}
        &\pi \circ G^{\widehat{n}_{k}+1}(\widehat{x}^{u}_{k}+x_{u}, \widehat{y}^{u}_{k}+y_{u}, x_{s}, y_{s}) \\
        &= \qty(- \lambda_{cu0}^{2 \widehat{n}_{k(0)}}\lambda_{cu1}^{2\widehat{n}_{k(1)}}x_{u}^{2} + \lambda_{*}^{-1}(\lambda_{s}^{\widehat{n}_{k}}x_{s}-a_{s})+ t_{k+1} + \widehat{x}^{s}_{k+1}+1,\ a_{2}\lambda_{cu0}^{\widehat{n}_{k(0)}}\lambda_{cu1}^{\widehat{n}_{k(1)}}x_{u} ).
    \end{align*}
    Since $t_{k+1} + \widehat{x}^{s}_{k+1} = \widehat{x}^{u}_{k+1}$, (1) follows.
    The same computation for the projection onto the $y_{u}y_{s}$-plane yields (2). See also \cite[Lemma 4.6]{KNS2023} for further details.
\end{proof}

Next, we prove lemmas \ref{con_pro1} and \ref{con_pro2} for the $x_{u}x_{s}$-plane.
For this purpose, we introduce the following notation.
For each $k > 0$, set
\begin{align*}
    \partial_{u}(W_{k}) &= W_{k} \cap \{x_{u} = \widehat{x}^{u}_{k} \pm b_{k}/2\}, \\
    \partial_{s} (W_{k}) &= W_{k} \cap \{x_{s} =  \pm x^{*}_{k}\},\\
    c(W_{k}) &= W_{k} \cap \{x_{u} = \widehat{x}^{u}_{k}\}.
\end{align*}
Under the projection $\pi \colon \mathbb{R}^{4} \to \mathbb{R}^{2}$ defined in \eqref{pro}, these sets correspond respectively to the two boundary edges of the rectangles $X_{k}$ parallel to the $x_{s}$- and $x_{u}$-axes and to its center line.
More precisely,
\begin{align*}
    \partial_{u}(X_{k}) &= \pi(\partial_{u}(W_{k})) = X_{k} \cap \{x_{u} = \widehat{x}^{u}_{k} \pm b_{k}/2\}, \\
    \partial_{s} (X_{k}) &= \pi(\partial_{s}(W_{k}))X_{k} \cap \{x_{s} =  \pm x^{*}_{k}\},\\
    c(X_{k}) &= \pi(c(W_{k})) = X_{k} \cap \{x_{u} = \widehat{x}^{u}_{k}\}.
\end{align*}

\begin{figure}[H]
  \centering
  \scalebox{0.3}{\includegraphics{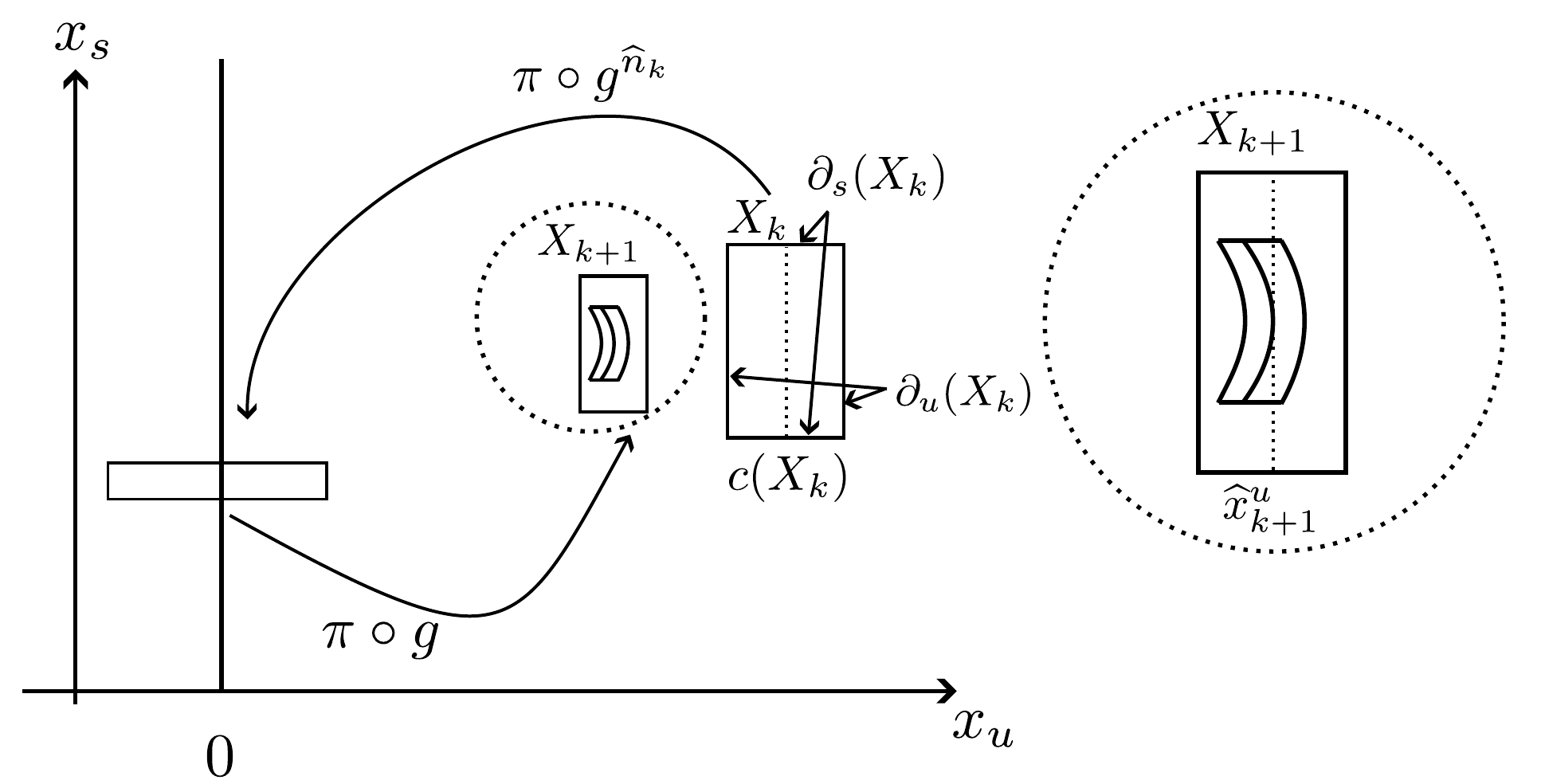}}
  \caption{}
  \label{rec_y}
\end{figure}

Let $\pi_{i} \colon \mathbb{R}^{4} \to \mathbb{R}$ $(i=1,2,3,4)$ be the projection onto the $i$-th coordinate, defined by
\begin{align*}
    \pi_{1}(x_{u}, y_{u}, x_{s}, y_{s}) = x_{u}, \quad \pi_{2}(x_{u}, y_{u}, x_{s}, y_{s}) = y_{u}, \\
    \pi_{3}(x_{u}, y_{u}, x_{s}, y_{s}) = x_{s}, \quad \pi_{4}(x_{u}, y_{u}, x_{s}, y_{s}) = y_{s}.
\end{align*}

\begin{lem}\label{con_pro1}
    There exists an integer $k_{1} \geq k_{0}$ such that, for every $k > k_{1}$,
    \begin{equation*}
        \pi_{1}(G^{\widehat{n}_{k}+1}(W_{k})) \subset \pi_{1}(W_{k+1}).
    \end{equation*}
\end{lem}
\begin{proof}
    From \eqref{2_f} the curve $\pi_{1}\circ G^{\widehat{n}_{k}+1}(\partial _{s}(W_{k}))$ is a parabola.
    The farthest point between $\pi (c(W_{k+1}))$ and the set $\pi \circ G^{\widehat{n}_{k}+1}(W_{k})$ is attained at one of the endpoints of this parabola.
    In particular, in \eqref{varphi}, this occurs at $x_{u} = b_{k}/2$ and $x_{s} = x^{*}_{k}$.
    Hence
    \begin{align*}
        \text{sup}\qty{d(t, c(X_{k+1})) \ | \ t \in \pi(G^{\widehat{n}_{k}+1}(W_{k}))} &= \qty| \qty(\lambda_{cu0}^{\widehat{n}_{k(0)}}\lambda_{cu1}^{\widehat{n}_{k(1)}} \frac{b_{k}}{2})^{2} + \lambda_{s}^{\widehat{n}_{k}} x^{*}_{k} | \\
        &= \qty| \frac{b_{k+1}}{4} + \lambda_{s}^{\widehat{n}_{k}} x^{*}_{k} |.
    \end{align*}
    To prove lemma \ref{con_pro1}, we estimate the following inequality:
    \begin{align}\label{ev1}
        \frac{\text{sup}\qty{d(t, c(X_{k+1})) \ | \ t \in \pi(G^{\widehat{n}_{k}+1}(W_{k}))}}{\text{sup}\qty{d(t, c(X_{k+1})) \ | \ t \in \partial _{u}(X_{k})}} \leq \frac{1}{2} + \qty|\frac{\lambda_{s}^{\widehat{n}_{k}}x^{*}_{k}}{b_{k+1}/2}|.
    \end{align}
    The second term on the right-hand side of \eqref{ev1} can be estimated as follows.
    \begin{align*}
        &\qty|\frac{\lambda_{s}^{\widehat{n}_{k}}x^{*}_{k}}{b_{k+1}/2}| \\
        &= \lambda_{s}^{\widehat{n}_{k}} 40  a_{2} b_{k}^{1/2} b_{k+1}^{-1} \\
        &= 40   a_{2} \lambda_{s}^{\widehat{n}_{k}} \qty(\lambda_{cu0}^{-\widehat{n}_{k(0)}}\lambda_{cu1}^{-\widehat{n}_{k(1)}})^{2} \qty( \lambda_{cu0}^{\sum_{i=0}^{\infty} \widehat{n}_{k+i(0)}/2^{i}} \lambda_{cu1}^{\sum_{i=0}^{\infty} \widehat{n}_{k+i(1)}/2^{i}})^{3/2}.
    \end{align*}
    Now, by lemma \ref{lem4_1},
    \begin{align*}
        \frac{3}{2} \sum_{i=0}^{\infty} \frac{\widehat{n}_{k+i(0)}}{2^{i}} \leq \frac{3}{2} \widehat{n}_{k(0)} \sum_{i=0}^{\infty} \qty(\frac{1+ \eta}{2})^{i} = \frac{3 \widehat{n}_{k(0)}}{1 - \eta} = (3+\eta_{1})\widehat{n}_{k(0)}
    \end{align*}
    where we set $\eta_{1} = 3\eta/(1 - \eta)$.
    Since the constants $\lambda_{cu0}$, $\lambda_{cu1}$, $\lambda_{s}$ satisfy $\lambda_{cu0}\lambda_{cu1}\lambda_{s} < 1$, we may choose $\eta >0$ sufficiently small so that
    \begin{equation*}
        \lambda_{cu0}^{(1+\eta_{1})}\lambda_{cu1}^{(1+\eta_{1})}\lambda_{s} < 1.
    \end{equation*}
    Moreover, using the majority condition \eqref{maj}, we obtain the following estimate:
    \begin{equation}\label{ev}
        \begin{split}
            \qty|\frac{\lambda_{s}^{\widehat{n}_{k}}x^{*}_{k}}{b_{k+1}/2}| &\leq 40 a_{2} \lambda_{s}^{\widehat{n}_{k}} \lambda_{cu0}^{-2\widehat{n}_{k(0)}}\lambda_{cu1}^{-2\widehat{n}_{k(1)}} \lambda_{cu0}^{3(1+ \eta_{1})\widehat{n}_{k(0)}}\lambda_{cu1}^{3(1+ \eta_{1})\widehat{n}_{k(1)}} \\
            &= 40 a_{2}\qty(\lambda_{cu0}^{(1+\eta_{1})}\lambda_{s})^{\widehat{n}_{k(0)}} \qty(\lambda_{cu1}^{(1+\eta_{1})}\lambda_{s})^{\widehat{n}_{k(1)}} \\
            &\leq 40 a_{2}\qty(\lambda_{cu0}^{(1+\eta_{1})}\lambda_{cu1}^{(1+\eta_{1})}\lambda_{s})^{\widehat{n}_{k(1)}} .
        \end{split}
    \end{equation}
    Since the right-hand side of \eqref{ev} tends to zero as $k \to \infty$, the right-hand side of \eqref{ev1} becomes strictly less than $1$.
    Thus, the desired estimate in lemma \ref{con_pro1} follows.
\end{proof}

\begin{lem}\label{con_pro2}
    There exists an integer $k_{1}^{\prime} \geq k_{0}$ such that, for every $k > k_{1}^{\prime}$,
    \begin{equation*}
        \pi_{3}(G^{\widehat{n}_{k}+1}(W_{k})) \subset \pi_{3}(W_{k+1}).
    \end{equation*}
\end{lem}
\begin{proof}
    To prove lemma \ref{con_pro2}, it suffices to estimate the distance between the endpoint of $\pi \circ G^{\widehat{n}_{k}+1}(\partial_{u} W_{k})$ and the line $\qty{x_{s} = \frac{1}{2}}$. By lemma \ref{con_psi_phi}, we have
    \begin{equation*}
        \pi_{3} \circ G^{\widehat{n}_{k} +1}(\widehat{x}^{u}_{k} + b_{k}/2, \widehat{y}^{u}_{k} + \bar{b}_{k}/2, 0, 0)= a_{2}\lambda_{cu0}^{\widehat{n}_{k(0)}}\lambda_{cu1}^{\widehat{n}_{k(1)}}b_{k}/2.
    \end{equation*}
    Thus, to obtain the desired inclusion, it is enough to verify that
    \begin{equation*}
        a_{2}\lambda_{cu0}^{\widehat{n}_{k(0)}}\lambda_{cu1}^{\widehat{n}_{k(1)}}\frac{b_{k}}{2} < x^{*}_{k+1}.
    \end{equation*}
    Indeed, by \eqref{4_con1} and \eqref{4_con2}, we obtain
    \begin{align*}
    x^{*}_{k+1} &= 20 a_{2} b^{\frac{1}{2}}_{k+1} \\
    &= 20 a_{2}\lambda_{cu0}^{\widehat{n}_{k(0)}} \lambda_{cu1}^{\widehat{n}_{k(1)}}b_{k} \\
    &> a_{2} \lambda_{cu0}^{\widehat{n}_{k(0)}} \lambda_{cu1}^{\widehat{n}_{k(1)}} \frac{b_{k}}{2}.
  \end{align*}
  This yields the desired inequality.
\end{proof}

The same arguments applied to the $y_{u}y_{s}$-plane give the following two lemmas.

\begin{lem}\label{con_pro3}
    There exists an integer $k_{2} \geq k_{0}$ such that, for every $k > k_{2}$,
    \begin{equation*}
        \pi_{2}(G^{\widehat{n}_{k}+1}(W_{k})) \subset \pi_{2}(W_{k+1}).
    \end{equation*}
\end{lem}

\begin{lem}\label{con_pro4}
    There exists an integer $k_{2}^{\prime} \geq k_{0}$ such that, for every $k > k_{2}^{\prime}$,
    \begin{equation*}
        \pi_{4}(G^{\widehat{n}_{k}+1}(W_{k})) \subset \pi_{4}(W_{k+1}).
    \end{equation*}
\end{lem}

Using these results, we now proceed to the proof of theorem \ref{main3}.

\begin{proof}[Proof of Theorem \ref{main3}]
    By lemma \ref{con_psi_phi}, \ref{con_pro1}, \ref{con_pro2}, \ref{con_pro3} and \ref{con_pro4}, let $K = \max\{k_{1},k^{\prime}_{1}, k_{2}, k^{\prime}_{2}\}$.
    Then, for every $k \geq K$,
    \begin{align*}
        G(W_{k}) \subset \text{int}(W_{k+1}).
    \end{align*}
    Moreover, from the definitions of $b_{k}$ and $\bar{b}_{k}$, we have
    \begin{align*}
        b_{k} &< (\lambda^{-1}_{cu1})^{2\widehat{n}_{k}} \rightarrow 0 \quad (k \rightarrow \infty), \\
        \bar{b}_{k} &< (1-2\lambda^{-1}_{u})(\lambda^{-1}_{u})^{2\widehat{n}_{k}} \rightarrow 0 \quad (k \rightarrow \infty).
    \end{align*}
    Recall that
    \begin{align*}
        x^{*}_{k} = 20a_{2}b^{1/2}_{k}, \quad y^{*}_{k} = 20a_{1}^{-1/2}\bar{b}^{1/2}_{k}.
    \end{align*}
    Hence, the diameter of $W_{k+1}$ satisfies
    \begin{equation*}
        \lim_{k \rightarrow \infty} \text{diam}(W_{k+1}) = 0.
    \end{equation*}
    This completes the proof of theorem \ref{main3}.
\end{proof}

\section{Statistical behavior}\label{sb}
Let $G$ be the diffeomorphism defined in \eqref{per_g}, and let $D_{k} = \text{Int}\, (W_{k})$ denote the contracting wandering domains of the map $G$ obtained in theorem \ref{main3}.
The code $\underline{\gamma}^{(k)}$ was given by
\begin{equation*}
    \underline{\gamma}^{(k)} = \underline{\alpha}^{(k)} \underline{u}^{(k)} [\underline{\omega}^{(k+1)}]^{-1}
\end{equation*}
Here, the code $\underline{u}^{(k)}$ is the part that can be chosen freely.
To prove (1) of theorem \ref{main}, we suppose the following condition.
\begin{itemize}
    \item \textbf{(Era condition)}
    Consider an increasing sequence of integers $(k_{s})_{s \in \mathbb{N}}$ satisfying the following:
    for every $s \in \mathbb{N}$,
    \begin{eqnarray*}
        \sum_{k=k_{s}}^{k_{s+1}-1}k^{2} > s \sum_{k=k_{1}}^{k_{s}-1}k^{2}.
    \end{eqnarray*}
    \item \textbf{(Code condition)}
    Under the Era condition above, for each integer $k \geq 1$, the code $\underline{u}^{(k)} = u(1)u(2)\dots u(k^{2})$ satisfies the following:
    \begin{itemize}
        \item [(1)] If $k_{s}\leq k < k_{s+1}$ with $s$ odd, then
        \begin{eqnarray*}\label{eq21}
            u_{i} =
            \begin{cases}
                0 \hspace{3mm}(i=1,\dots ,\lfloor 3k^{2}/4 \rfloor), \\
                1 \hspace{3mm}(i=\lfloor 3k^{2}/4 \rfloor +1, \dots ,k^{2}).
            \end{cases}
        \end{eqnarray*}
        That is,
        \begin{equation*}
            \underline{u}^{(k)} = \overbrace{000\dots 0}^{\lfloor 3k^{2}/4 \rfloor}\overbrace{1 \dots 1}^{\lceil k^{2}/4 \rceil}.
        \end{equation*}
        \item [(2)] If $k_{s}\leq k < k_{s+1}$ with $s$ even, then
        \begin{eqnarray*}\label{eq22}
            u_{i} =
            \begin{cases}
                0 \hspace{3mm}(i=1,\dots ,\lfloor 7k^{2}/8 \rfloor), \\
                1 \hspace{3mm}(i=\lfloor 7k^{2}/8 \rfloor +1, \dots ,k^{2}).
            \end{cases}
        \end{eqnarray*}
        That is,
        \begin{equation*}
            \underline{u}^{(k)} = \overbrace{000\dots 0}^{\lfloor 7k^{2}/8 \rfloor}\overbrace{1 \dots 1}^{\lceil k^{2}/8 \rceil}.
        \end{equation*}
    \end{itemize}
\end{itemize}
Under these conditions, one can show that $G$ has historic contracting wandering domain.
For further details, see \cite[Theorem 5.1]{KNS2023}.

To prove (2) of theorem \ref{main}, we set
\begin{eqnarray*}
    \underline{u}^{(k)} = \overbrace{000\dots 00}^{k^{2}}.
\end{eqnarray*}
With this choice, one can show that $G$ has a contracting wandering domain such that the sequence \eqref{eq1} converges to the periodic measure supported in the saddle fixed point $p_{g}$.
See \cite[Theorem 5.5]{KNS2023} for detail calculations.

\subsection*{Acknowledgements}
The author would like to express their gratitude to Masato Tsujii for fruitful discussions.
The author is also grateful to Shin Kiriki, Yushi Nakano and Teruhiko Soma for suggesting the problem and for valuable comments.
The author further thanks Masayuki Asaoka for helpful discussions and insightful comments on the content of this work.
The author was supported by WISE program (MEXT) at Kyushu University.

\subsection*{Use of AI tools}
The author used ChatGPT only for assistance with improving the English phrasing of the manuscript; no mathematical arguments, proofs, or scientific content were generated or influenced by it.
\bibliographystyle{amsalpha} \bibliography{ref}

\bigskip
{\scshape Joint graduate school of mathematics for innovation, Kyushu University, 744 Motooka, Nishi-ku, Fukuoka, 819-0395, Japan} \\
\textit{Email address}: \texttt{yamamoto.kodai.508@s.kyushu-u.ac.jp}
\end{document}